\documentclass[a4paper]{article}

\usepackage[latin1]{inputenc}  
\usepackage{graphicx}
\usepackage[pdftex,bookmarks]{hyperref}
\usepackage{amsmath}    
\usepackage{amssymb}
\usepackage{amsmath,amsthm}
\usepackage{textcomp}
\usepackage[all]{xy}

\newtheorem{teo}{Theorem}
\newtheorem{defi}{Definition}
\newtheorem{lemma}{Lemma}

\newtheorem{cor}{Corollary}
\newtheorem{prop}{Proposition}
\newtheorem{exa} {Example}
\newtheorem{rem} {Remark}

\DeclareMathOperator{\im}{im}
\DeclareMathOperator{\ind}{ind}

\title{\huge Poincar\'e duality, Hilbert complexes and   geometric applications}
\author{Francesco Bei}
\date{}

\begin{document}

\maketitle
 \pagestyle{myheadings} \markboth {\centerline{\small \tt Francesco Bei}}{\centerline{\small \tt Poincar\'e duality, Hilbert complexes and   geometric applications}}

\bigskip

\begin{abstract}
Let $(M,g)$ be an open, oriented and incomplete riemannian manifold. The aim of this paper is to study  the  following two sequences of $L^2-$cohomology groups:
\begin{enumerate}
\item $H^i_{2,m\rightarrow M}(M,g)$  defined as the image $(H^i_{2,min}(M,g)\rightarrow H^i_{2,max}(M,g))$ 
\item $\overline{H}^i_{2,m\rightarrow M}(M,g)$ defined as the image $(\overline{H}^i_{2,min}(M,g)\rightarrow \overline{H}^i_{2,max}(M,g))$. 
\end{enumerate}We show, under suitable hypothesis, that the first sequence is the cohomology of a  Hilbert complex which contains the minimal one and is contained in the maximal one.  In particular this leads us to prove a Hodge theorem for these groups. We also show that when the second sequence is finite dimensional then  Poincar\'e duality holds  and that, with the same assumptions,  when $dim(M)=4n$ then   we can employ $\overline{H}^{2n}_{2,m\rightarrow M} (M,g)$ in order to define an $L^2-$signature on $M$. We prove several applications to the intersection cohomology of compact smoothly stratified pseudomanifolds and  we get some results about the Friedrichs extension  $\Delta_{i}^\mathcal{F}$ of $\Delta_{i}$.
\end{abstract}


\section*{Introduction}
The study of singular spaces from a geometric differential point of view leads naturally to consider (and  study) open differentiable manifolds with incomplete riemannian metric. A great variety of papers have been devoted, for example, to the relationship between the $L^2$ Hodge and de Rham cohomology  associated with incomplete metrics on $M$, the smooth part of a compact stratified pseudomanifold $X$,  and the intersection cohomology of $X$ with respect to suitable perversities. We mention here, without any claim of completeness, the classic paper of Cheeger \cite{C}, the papers of Nagase \cite{MN}, \cite{MAN}, the paper of Hunsicker and  Hunsicker and Mazzeo \cite{H} and \cite{HM} and the paper of Saper \cite{LS}. \\Nevertheless, as it is well known, when $M$ is an open manifold and $g$ is an incomplete riemannian metric on $M$ then the de Rham differential $d_i:L^2\Omega^{i}(M,g)\rightarrow L^2\Omega^{i+1}(M,g)$ could have many closed extensions when we look at it as an unbounded operator defined over the smooth forms with compact support. This implies that there exists several ways to turn the complex $(\Omega^i_{c}(M),d_{i})$ into a Hilbert complex and perhaps the most natural ones are  $(L^2\Omega^i(M,g),d_{max,i})$ and $(L^2\Omega^i(M,g),d_{min,i})$ where $d_{max,i}:L^{2}\Omega^{i}(M,g)\rightarrow L^{2}\Omega^{i+1}(M,g)$ is defined in the distributional sense and $d_{min,i}:L^{2}\Omega^{i}(M,g)\rightarrow L^{2}\Omega^{i+1}(M,g)$ is defined as the closure, under the graph norm, of $d_{i}:\Omega_{c}^i(M)\rightarrow \Omega^{i+1}_{c}(M)$. So a natural and fundamental  question is:
\begin{itemize}
\item if these two Hilbert complexes have finite dimensional $L^{2}-$cohomology groups or finite dimensional reduced $L^2-$cohomology groups, does Poincar\'e duality hold for them?
\end{itemize}
 As it is well known the answer is usually negative in both cases.\\ Anyway, from the pair of Hilbert complexes $(L^{2}\Omega^i(M,g),d_{max/min,i})$ we can get other sequences of $L^2-$cohomology groups defined in the following way:
\begin{equation}
\label{gfhg}
H^i_{2,m\rightarrow M}(M,g),\ \text{defined as the image}\ H^i_{2,min}(M,g)\longrightarrow H^i_{2,max}(M,g)
\end{equation}
\begin{equation}
\label{fghgf}
\overline{H}^i_{2,m\rightarrow M}(M,g),\  \text{defined as the image}\ \overline{H}^i_{2,min}(M,g)\longrightarrow \overline{H}^i_{2,max}(M,g).
\end{equation}
where  in \eqref{gfhg}, as well as in  \eqref{fghgf},  the map is the map induced in cohomology (reduced cohomology) by  the natural inclusion of complexes $(L^{2}\Omega^i(M,g),d_{min,i})\subseteq (L^{2}\Omega^i(M,g),d_{max,i})$.

At this point we can summarize the goal of this paper in the following way:
\begin{itemize}
\item Investigate the properties of the groups, $H^i_{2,m\rightarrow M}(M,g)$, $\overline{H}^i_{2,m\rightarrow M}(M,g)$ $i=0,...,dimM$. More precisely we show that, under certain assumptions, the first sequence is the cohomology of a Hilbert complex which contains the minimal one and is contained in the maximal one. In particular this leads us to prove a Hodge theorem for the groups $H^i_{2,m\rightarrow M}(M,g)$. For the sequence defined in \eqref{fghgf} we show  that, when it is finite dimensional, then  Poincar\'e duality holds for it. In particular, combining these two properties, we will show that when $(L^2\Omega^i(M,g),d_{max/min,i})$ is a Fredholm complex then $H^i_{2,m\rightarrow M}(M,g)$  is the cohomology of a Fredholm complex for which Poincar\'e duality holds; these should be regarded as the main results of this paper.  Then we show that, when $dim(M)=4n$,  we can use $\overline{H}^{2n}_{2,m\rightarrow M}(M,g)$ in order  to define an $L^2-$signature on $M$ and to get the existence of a topological signature on $M$. Moreover we show several applications to stratified pseudomanifolds and  we get a topological obstruction to the existence of a riemannian metric (complete or incomplete) with finite $L^2-$cohomology (reduced and unreduced). Finally we show some applications to $\Delta_{i}^\mathcal{F}$, the Friedrichs extension of $\Delta_{i}$; in particular we prove that when $(L^2\Omega^i(M,g),d_{max/min,i})$ are Fredholm complexes, then $\Delta_{i}^\mathcal{F}$ is a Fredholm operator for each $i$. This last result applies, for example, when $M$ is the regular part of a compact and smoothly stratified pseudomanifold with a Thom-Mather stratification.
\end{itemize}
\medskip

\noindent
The paper is structured in the following way:\\in the first section we introduce the notion of Hilbert complexes; we generalize to this abstract framework the properties of the pair $(L^2\Omega^i(M,g),d_{min,i})\subset (L^2\Omega^i(M,g),d_{max,i})$. In particular in Definition \ref{edera} we introduce the notion of complementary Hilbert complexes, that is a pair of Hilbert complexes $(H_{i},D_{i})\subseteq (H_{i},L_{i})$ such that for each $i$ there exists an isometry $\phi_{i}:H_{i}\rightarrow H_{n-i}$ which satisfies $\phi_{i}(\mathcal{D}(D_i))=\mathcal{D}(L^*_{n-i-1})$ and $L_{n-i-1}^*\circ\phi_{i}=C_{i}(\phi_{i+1}\circ D_{i})$ on $\mathcal{D}(D_{i})$,
 where $ L^{*}_{n-i-1}:H_{n-i}\rightarrow H_{n-i-1}$ is the adjoint of $L_{n-i-1}:H_{n-i-1}\rightarrow H_{n-i}$ and $C_i\neq0$ is a constant which depends only on $i$.  Then we prove these two theorems:
\begin{teo}
Let $(H_{j},D_{j})\subseteq (H_{j},L_{j})$ be a pair of complementary Hilbert complexes. Let $i_{r,j}^{*}$ be the map induced by the inclusion of complexes between the reduced cohomology groups. 
Suppose that for each $j$ 
\begin{equation}
\im(\overline{H}^{j}(H_{*},D_{*})\stackrel{i^{*}_{r,j}}{\longrightarrow}\overline{H}^{j}(H_{*},L_{*}))
\end{equation}
is finite dimensional. Then 
\begin{equation}
\im(\overline{H}^{j}(H_{*},D_{*})\stackrel{i^{*}_{r,j}}{\longrightarrow}\overline{H}^{j}(H_{*},L_{*})),\  j=0,...,n
\label{kiollk}
\end{equation}
is a finite sequence of finite dimensional  vector spaces with Poincar\'e duality.
\label{dualityq}
\end{teo}

The second theorem describes an abstract framework in which the groups $\im(H^i(H_*,D_*) \rightarrow H^i(H_*,L_*))$ are effectively the cohomology groups of a Hilbert complex which is intermediate between $(H_{j},D_{j})$ and $(H_{j},L_{j})$:

\begin{teo}
\label{cheaffa}
Let  $(H_{j},D_{j})\subseteq (H_{j},L_{j})$  be  a pair of  Hilbert complexes. Suppose that for each $j$ $ran(D_{j})$ is closed in $H_{j+1}$. Then there exists a third Hilbert complex $(H_{j},P_{j})$ such that 
\begin{enumerate}
\item  $(H_{j},D_{j})\subseteq (H_{j},P_{j})\subseteq (H_{j},L_{j})$ 
\item $H^i(H_{*},P_*)=\im(H^i(H_*,D_*) \rightarrow H^i(H_*,L_*))$. 
\end{enumerate}
Moreover if $(H_{j},D_{j})\subseteq (H_{j},L_{j})$ are complementary and  $(H_{j},D_{j})$ or equivalently $(H_{j},L_{j})$ is Fredholm then $(H_{j},P_{j})$ is a  Fredholm complex with Poincar\'e duality.
\end{teo}
\vspace{1 cm}

In the second section we specialize the situation to the pair of complementary Hilbert complexes that are natural in riemannian geometry; our main results are the following two theorems which are a consequence of the two previous results:

\begin{teo}
\label{marione}
Let $(M,g)$ be an open, oriented and incomplete riemannian manifold of dimension $m$. Then the complexes $$(L^{2}\Omega^{i}(M,g),d_{max,i})\ and\  (L^{2}\Omega^{i}(M,g),d_{min,i})$$ are a pair of complementary Hilbert complexes. \\In particular if $\im(\overline{H}^{i}_{2,min}(M,g)\stackrel{i^{*}_{r,i}}{\longrightarrow}\overline{H}^{i}_{2,max}(M,g))$ is finite dimensional for each $i$ then $$\im(\overline{H}^{i}_{2,min}(M,g)\stackrel{i^{*}_{r,i}}{\longrightarrow}\overline{H}^{i}_{2,max}(M,g))$$ is a finite sequence of finite dimensional vector spaces with Poincar\'e duality.
\end{teo}
 
Another application of Theorem \ref{cheaffa} gives the following :
\begin{teo}
\label{polcass}
Let $(M,g)$ be an open, oriented and incomplete  riemannian manifold of dimension $n$. Suppose that for each $i$ $ran(d_{min,i})$ is closed in $L^2\Omega^{i+1}(M,g)$. Then there exists a Hilbert complex $(L^2\Omega^i(M,g)),d_{\mathfrak{m},i})$ such that for each $i=0,...n$ $$\mathcal{D}(d_{min,i})\subset \mathcal{D}(d_{\mathfrak{m},i})\subset \mathcal{D}(d_{max,i}),$$$d_{max,i}$ is an extension of $d_{\mathfrak{m},i}$ which is an extension of $d_{min,i}$ and   $$H^i_{2,\mathfrak{m}}(M,g)=\im(H^{i}_{2,min}(M,g)\stackrel{i^{*}_{i}}{\longrightarrow}H^{i}_{2,max}(M,g))$$ where $H^i_{2,\mathfrak{m}}(M,g)$ is the cohomology of the Hilbert complex $(L^2\Omega^i(M,g),d_{\mathfrak{m},i})$. Finally, if\\ $(L^2\Omega^i(M,g),d_{max,i})$ or equivalently $(L^2\Omega^i(M,g),d_{min,i})$ is Fredholm, then $(L^2\Omega^i(M,g),d_{\mathfrak{m},i})$ is a Fredholm complex with Poincar\'e duality.
\end{teo}

From the previous theorem we get as corollary that under certain conditions  it is possible to construct a self-adjoint extension of $\Delta_{i}:\Omega^i_{c}(M)\rightarrow \Omega_{c}^i(M)$, the Laplacian acting on the space of smooth compactly supported $i-$forms, such that it is a Fredholm operator with nullspace isomorphic to $\im(H^{i}_{2,min}(M,g)\stackrel{i^{*}_{i}}{\longrightarrow}H^{i}_{2,max}(M,g))$. In other words it is possible to state (and prove)  a \textbf{Hodge theorem} for the cohomology groups  $\im(H^{i}_{2,min}(M,g)\stackrel{i^{*}_{i}}{\longrightarrow}H^{i}_{2,max}(M,g))$:

\begin{cor}
\label{giover}
In the same assumptions of Theorem \ref{polcass}; Let $\Delta_{i}:\Omega^i_{c}(M)\rightarrow \Omega_{c}^i(M)$ be the Laplacian acting on the space of smooth compactly supported forms. Then there exists a self-adjoint extension  $\Delta_{\mathfrak{m},i}:L^2\Omega^i(M,g)\rightarrow L^2\Omega^i(M,g)$ with closed range such that $$Ker(\Delta_{\mathfrak{m},i})\cong \im(H^{i}_{2,min}(M,g)\stackrel{i^{*}_{i}}{\longrightarrow}H^{i}_{2,max}(M,g)).$$ Moreover, if $(L^2\Omega^i(M,g),d_{max,i})$ or equivalently  $(L^2\Omega^i(M,g),d_{min,i})$ is Fredholm, then  $\Delta_{\mathfrak{m},i}$ is a Fredholm operator on its domain endowed with the graph norm.
\end{cor}

In the rest of the section we prove several applications of these results; in particular we show that there might exist   topological obstructions to the existence of a riemannian metric with finite $L^2-$cohomology  groups, see Corollary \ref{gianni}.

We end the second sections by showing that, when $(M,g)$ is an open, oriented and incomplete riemannian manifold of dimension $4n$ such that $\im(\overline{H}^{2n}_{2,min}(M,g)\rightarrow \overline{H}^{2n}_{2,max}(M,g))$ is finite dimensional, then it is possible to define an \textbf{$L^2-$signature} on $M$ and that this implies also the existence of a \textbf{topological signature} on $M$; see Definition \ref{qwerq} and Prop. \ref{topomicio}.\vspace{1 cm}

The third section is devoted to the applications of the previous results to compact smoothly  stratified pseudomanifold with a Thom-Mather stratification; after recalling the $L^2-$Hodge-de Rham theorem proved in \cite{FB}, we get some consequences for the intersection cohomology groups associated with some general perversity in the sense of Friedman; see Proposition \ref{derede}, Corollaries \ref{frengoesto}, \ref{gennaron}, \ref{mummio} and \ref{falco}.
In particular we have the following \textbf{Hodge and index theorems}:
\begin{teo}
\label{tizioss}
Let $X$ be a compact smoothly stratified pseudomanifold of dimension $n$ with a Thom-Mather stratification.  Let $g$ be a quasi-edge metric with weights on $reg(X)$, see Def. \ref{zedge}.  Then    we have the following results:
\begin{equation}
\label{miopf}
Ker(\Delta_{\mathfrak{m},i})\cong \im(I^{q_{g}}H^i(X,\mathcal{R}_0)\rightarrow I^{p_{g}}H^i(X,\mathcal{R}_0))
\end{equation}
\begin{equation}
\label{ssssa}
\ind((d_{\mathfrak{m}}+d_{\mathfrak{m}}^*)_{ev})=I^{p_{g}\rightarrow q_{g}}\chi(X,\mathcal{R}_0)
\end{equation}
where $I^{p_{g}\rightarrow q_{g}}\chi(X,\mathcal{R}_0)=\sum_{i}(-1)^idim(\im(I^{q_{g}}H^i(X,\mathcal{R}_0)\rightarrow I^{p_{g}}H^i(X,\mathcal{R}_0)))$ and  $(d_{\mathfrak{m}}+d_{\mathfrak{m}}^*)_{ev}$ is the  extension of $$d+\delta:\bigoplus_i\Omega_c^{2i}(M)\rightarrow \bigoplus_i\Omega_c^{2i+1}(M)\  defined\ by\  (d_{\mathfrak{m}}+d_{\mathfrak{m}}^*)_{ev}|_{L^2\Omega^{2i}(M,g)}:=d_{\mathfrak{m},2i}+d_{\mathfrak{m},2i-1}^*$$ which is a Fredholm operator on its domain endowed with the graph norm.
\end{teo}
 Moreover we remark that in this framework  the $L^2$ signature introduced in the previous section in a more general context has a topological meaning  because it coincides with the \textbf{perverse signature} introduced by Friedman and Hunsicker in \cite{FH}, that is $$\sigma_{2}(reg(X),g)=\sigma_{q_g\rightarrow p_g}(X).$$

Finally in the last section we show some applications to $\Delta_{i}^\mathcal{F}$, the Friedrichs extension of $\Delta_{i}$. Our main result is:
\begin{teo} Let $(M,g)$ be an open, oriented and incomplete riemannian manifold such that $(L^2\Omega^i(M,g),d_{max,i})$, or equivalently $(L^2\Omega^i(M,g),d_{min,i})$, is a Fredholm complex. Then for each $i$, $\Delta_{i}^{\mathcal{F}}$, the Friedrichs extension of $\Delta_i:\Omega_{c}^{i}(M)\rightarrow \Omega^i_{c}(M)$, is a Fredholm operator on its domain endowed with the graph norm. 
\end{teo}
As a particular case of the previous theorem we have the following corollary:
\begin{cor}
\label{ilbossrussa}
Let $X$ be a compact smoothly and oriented stratified pseudomanifold of dimension $n$ with a Thom-Mather stratification. Let $g$ be a quasi-edge metric with weights on $reg(X)$. Then on $L^{2}\Omega^i(reg(X),g)$, for each $i=0,...,n$, $\Delta_{i}^\mathcal{F}$ is a Fredholm operator on its domain endowed with the graph norm.
\end{cor}
\vspace{1 cm}

\textbf{Acknowledgments.} This paper is part of my PhD thesis developed at the Department of Mathematics of
Sapienza, University of Rome. I wish to thank my advisor, Paolo Piazza, for having suggested this subject, for
his help and for many helpful discussions.  I wish also to thank Pierre Albin for having invited me to spend the months of March and April 2012 at  the University of Illinois at Urbana-Champaign, for many interesting discussions and for many helpful hints. Finally I wish to thank the anonymous referee  for many helpful comments and suggestions.

\section {Hilbert Complexes}

We start the section recalling the notion of Hilbert complex and its main properties.  For a complete development of the subject we refer to \cite{BL}.

\begin{defi} A Hilbert complex is a complex, $(H_{*},D_{*})$ of the form:
\begin{equation}
0\rightarrow H_{0}\stackrel{D_{0}}{\rightarrow}H_{1}\stackrel{D_{1}}{\rightarrow}H_{2}\stackrel{D_{2}}{\rightarrow}...\stackrel{D_{n-1}}{\rightarrow}H_{n}\rightarrow 0,
\label{mm}
\end{equation}
where each $H_{i}$ is a separable Hilbert space and each map $D_{i}$ is a closed operator called the differential such that:
\begin{enumerate}
\item $\mathcal{D}(D_{i})$, the domain of $D_{i}$, is dense in $H_{i}$.
\item $ran(D_{i})\subset \mathcal{D}(D_{i+1})$.
\item $D_{i+1}\circ D_{i}=0$ for all $i$.
\end{enumerate}
\end{defi}
The cohomology groups of the complex are $H^{i}(H_{*},D_{*}):=Ker(D_{i})/ran(D_{i-1})$. If the groups $H^{i}(H_{*},D_{*})$ are all finite dimensional we say that it is a  $Fredholm\ complex$. 

Given a Hilbert complex there is a dual Hilbert complex
\begin{equation}
0\leftarrow H_{0}\stackrel{D_{0}^{*}}{\leftarrow}H_{1}\stackrel{D_{1}^{*}}{\leftarrow}H_{2}\stackrel{D_{2}^{*}}{\leftarrow}...\stackrel{D_{n-1}^{*}}{\leftarrow}H_{n}\leftarrow 0,
\label{mmp}
\end{equation}
defined using $D_{i}^{*}:H_{i+1}\rightarrow H_{i}$, the Hilbert space adjoints of the differentials\\ $D_{i}:H_{i}\rightarrow H_{i+1}$. The cohomology groups of $(H_{j},(D_{j})^*)$, the dual Hilbert  complex, are $$H^{i}(H_{j},(D_{j})^*):=Ker(D_{n-i-1}^{*})/ran(D_{n-i}^*).$$\\For all $i$ there is also a laplacian $\Delta_{i}=D_{i}^{*}D_{i}+D_{i-1}D_{i-1}^{*}$ which is a self-adjoint operator on $H_{i}$ with domain 
\begin{equation}
\label{saed}
\mathcal{D}(\Delta_{i})=\{v\in \mathcal{D}(D_{i})\cap \mathcal{D}(D_{i-1}^{*}): D_{i}v\in \mathcal{D}(D_{i}^{*}), D_{i-1}^{*}v\in \mathcal{D}(D_{i-1})\}
\end{equation} and nullspace: 
\begin{equation}
\label{said}
\mathcal{H}^{i}(H_{*},D_{*}):=ker(\Delta_{i})=Ker(D_{i})\cap Ker(D_{i-1}^{*}).
\end{equation}

The following propositions are  standard results for these complexes. The first result is a weak Kodaira decomposition:

\begin{prop}
\label{beibei}
 [\cite{BL}, Lemma 2.1] Let $(H_{i},D_{i})$ be a Hilbert complex and $(H_{i},(D_{i})^{*})$ its dual complex, then: $$H_{i}=\mathcal{H}^{i}\oplus\overline{ran(D_{i-1})}\oplus\overline{ran(D_{i}^{*})}.$$
\label{kio}
\end{prop}

The reduced cohomology groups of the complex are: $$\overline{H}^{i}(H_{*},D_{*}):=Ker(D_{i})/(\overline{ran(D_{i-1})}).$$
By the above proposition there is a pair of  weak de Rham isomorphism theorems:
 \begin{equation}
\label{pppp}
\left\{
\begin{array}{ll}
\mathcal{H}^{i}(H_{*},D_{*})\cong\overline{H}^{i}(H_{*},D_{*})\\
\mathcal{H}^{i}(H_{*},D_{*})\cong\overline{H}^{n-i}(H_{*},(D_{*})^{*})
\end{array}
\right.
\end{equation}
where in the second case we mean the cohomology of the dual Hilbert complex.\\
The complex $(H_{*},D_{*})$ is said $ weak\  Fredholm$  if $\mathcal{H}_{i}(H_{*},D_{*})$ is finite dimensional for each $i$. By the next propositions it follows immediately that each Fredholm complex is a weak Fredholm  complex.

\begin{prop}
\label{topoamotore}
[\cite{BL}, corollary 2.5] If the cohomology of a Hilbert complex $(H_{*}, D_{*})$ is finite dimensional then, for all $i$,  $ran(D_{i-1})$ is closed  and  $H^{i}(H_{*},D_{*})\cong \mathcal{H}^{i}(H_{*},D_{*}).$
\end{prop}

\begin{prop}[\cite{BL}, corollary 2.6] A Hilbert complex  $(H_{j},D_{j}),\ j=0,...,n$ is a Fredholm complex (weak Fredholm) if and only if  its dual complex, $(H_{j},D_{j}^{*})$, is Fredholm (weak Fredholm). If it is Fredholm then
\begin{equation}
\mathcal{H}_{i}(H_{j},D_{j})\cong H_{i}(H_{j},D_{j})\cong H_{n-i}(H_{j},(D_{j})^{*})\cong \mathcal{H}_{n-i}(H_{j},(D_{j})^{*}).
\end{equation}
Analogously in the  the weak Fredholm case we have:
\begin{equation}
\mathcal{H}_{i}(H_{j},D_{j})\cong\overline{ H}_{i}(H_{j},D_{j})\cong\overline{ H}_{n-i}(H_{j},(D_{j})^{*})\cong \mathcal{H}_{n-i}(H_{j},(D_{j})^{*}).
\end{equation}
\label{fred}
\end{prop}

Now we recall another result which  shows that it is possible to compute the cohomology groups of a Hilbert complex using a core subcomplex $$\mathcal{D}^{\infty}(H_{i})\subset H_{i}.$$ For all $i$ we define  $\mathcal{D}^{\infty}(H_{i})$ as consisting  of all elements $\eta$ that are in the domain of $\Delta_{i}^{l}$ for all $l\geq 0.$

\begin{prop}[\cite{BL}, Theorem 2.12] The complex $(\mathcal{D}^{\infty}(H_{i}),D_{i})$ is a subcomplex quasi-isomorphic to the complex $(H_{i},D_{i})$.
\label{dfff}
\end{prop}

Given a pair of Hilbert complexes $(H_{j},D_{j})$ and $(H_{j},D'_{j})$ we will write $(H_{j},D_{j})\subseteq  (H_{j},D'_{j}) $ if for each $j$ one of the two following properties is satisfied:
\begin{enumerate}
\item $D'_{j}:H_{i}\rightarrow H_{j+1}$ is equal to   $D_{j}:H_{j}\rightarrow H_{j+1}$
\item $D'_{j}:H_{j}\rightarrow H_{j+1}$ is a proper closed extension of   $D_{j}:H_{i}\rightarrow H_{j+1}$
\end{enumerate}

We will write $(H_{j},D_{j})\subset  (H_{j},D'_{j})$ when the second of the above properties is satisfied. For each $j$ let $i_{j}:\mathcal{D}(D_{j})\rightarrow \mathcal{D}(L_{j})$ denote the natural inclusion of the domain of $D_{j}$ into the domain of $L_{j}$. Obviously $i_{j}$ induces a maps between $H^{j}(H_{*},D_{*})$ and $H^{j}(H_{*},L_{*}) $ and between  $\overline{H}^{j}(H_{*},D_{*})$ and $\overline{H}^{j}(H_{*},L_{*}) $. We will label the first as
\begin{equation}
 i_{j}^{*}:H^{j}(H_{*},D_{*})\rightarrow H^{j}(H_{*},L_{*})
\label{ytty}
\end{equation}
 and the second as
\begin{equation}
 i_{r,j}^{*}:\overline{H}^{j}(H_{*},D_{*})\rightarrow \overline{H}^{j}(H_{*},L_{*})
\label{oppio}
\end{equation}
Consider again a pair of Hilbert complexes $(H_{i},D_{i})$ and $(H_{i},L_{i})$ with $i=0,...n$ .
\begin{defi} 
\label{edera}
The pair  $(H_{i},D_{i})$ and $(H_{i},L_{i})$ is said to be \textbf{ related} if the following property is satisfied
\begin{itemize}
\item for each $i$ there exists a linear, continuous and bijective map $\phi_{i}:H_{i}\rightarrow H_{n-i}$ such that $\phi_i(\mathcal{D}(D_i))=\mathcal{D}(L^*_{n-i-1})$ and  $L_{n-i-1}^*\circ\phi_{i}=C_{i}(\phi_{i+1}\circ D_{i})$ on $\mathcal{D}(D_{i})$
 where $ L^{*}_{n-i-1}:H_{n-i}\rightarrow H_{n-i-1}$ is the adjoint of $L_{n-i-1}:H_{n-i-1}\rightarrow H_{n-i}$  
 and $C_{i}\neq 0$ is a constant which depends only on $i$.
\end{itemize}
Furthermore we call the maps $\phi_{i}$  \textbf{duality maps}; (later it will be clear why we choose this name). 
\label{qqqq}
\begin{itemize} 
\item We  call the complexes \textbf{complementary} if each $\phi_{i}$ is an isometry   between $H_{i}$ and $H_{n-i}$. 
\end{itemize}
\end{defi}

We have the following propositions:
\begin{prop}
\label{errew}
Let $(H_{i},D_{i})$ and $(H_{i},L_{i})$ be related Hilbert complexes. Then:
\begin{enumerate}
\item Also  $(H_{i},L_{i})$ and  $(H_{i},D_{i})$ are related  Hilbert complexes. Moreover if $\{\phi_{i}\}$ are the duality maps which make $(H_{i},D_{i})$ and $(H_{i},L_{i})$ related  then  $\{\phi_{i}^{*}\}$, the family of respective adjoint maps,  are the duality maps which make $(H_{i},L_{i})$ and $(H_{i},D_{i})$  related.
\item The complexes $(H_{i},D_{i})$ and $(H_{i},L_{i}^{*})$ have isomorphic cohomology groups and isomorphic reduced cohomology groups. In the same way the complexes $(H_{i},L_{i})$ and $(H_{i},D_{i}^{*})$ have isomorphic cohomology groups and isomorphic reduced cohomology groups.
\item The following isomorphisms hold: $\mathcal{H}^{j}(H_{i},D_{i})\cong \mathcal{H}^{n-j}(H_{i},L_{i}) $, $\overline{H}^{j}(H_{i},D_{i})\cong \overline{H}^{n-j}(H_{i},L_{i}).$
\item If the  complexes $(H_{i},D_{i})$ and $(H_{i},L_{i}^{*})$ are complementary then each $\phi_{j}$ induces an isomorphism between $\mathcal{H}^{j}(H_{i},D_{i})$ and $\mathcal{H}^{n-j}(H_{i},L_{i})$.
\end{enumerate}
\end{prop}

\begin{proof} By Definition \ref{qqqq} we know  that  $\phi_{i}^{*}:H_{n-i}\rightarrow H_{i}$, the adjoint of $\phi_{i}:H_{i}\rightarrow H_{n-i}$ , is a family of linear continuous and bijective maps. In this way if we look at  $ L^{*}_{n-i-1}\circ\phi_{i}$ as an unbounded linear map between $H_{i}$ and $H_{n-i-1}$ with domain $\mathcal{D}( L^{*}_{n-i-1}\circ\phi_{i})=\phi_{i}^{-1}(\mathcal{D}(L_{n-i-1}^*))$ $=\mathcal{D}(D_i)$ we have that   $( L^{*}_{n-i-1}\circ\phi_{i})^*=\phi_{i}^{*}\circ L_{n-i-1}$   that is the adjoint of $ L^{*}_{n-i-1}\circ\phi_{i}$ is $\phi_{i}^{*}\circ L_{n-i-1}$ with $\mathcal{D}(\phi_{i}^{*}\circ L_{n-i-1})=\mathcal{D}(L_{n-i-1})$.\\ In the same way we have $(\phi_{i+1}\circ D_{i})^{*}=(D_{i}^{*}\circ \phi_{i+1}^{*})$ where $\mathcal{D}( \phi_{i+1}\circ D_{i})=\mathcal{D}(D_{i})$ and $\mathcal{D}(D_{i}^*\circ \phi_{i+1}^*)=(\phi_{i+1}^{*})^{-1}(\mathcal{D}(D_{i}^*))$.  In this way 
we have that, for each $i$,  $\mathcal{D}(D_{i}^{*}\circ \phi_{i+1}^{*})=\mathcal{D}(\phi_i^*\circ L_{n-i-1})$, $C_i(D_{i}^{*}\circ \phi_{i+1}^*)=\phi_{i}^*\circ L_{n-i-1}$ on $\mathcal{D}(L_{n-i-1})$  and that $\phi_{i+1}^*(\mathcal{D}(L_{n-i-1}))=\mathcal{D}(D^*_i)$. So we can conclude that   the complexes $(H_{i},L_{i})$ and $(H_{i},D_{i})$ are  related with $\{\phi_{i}^{*}\}$ as duality maps.\\
 The second  property is an immediate consequence of definition \ref{qqqq} and the first point of the proposition . Now if we compose the isomorphisms of the second point with  the  isomorphisms of \eqref{pppp} we can get the isomorphisms of the third point. Finally if each $\phi_{i}$ is an isometry then $\phi_{i}^{*}=\phi_{i}^{-1}$. By Definition \ref{qqqq}  we know that $\phi_{i}$ induces an isomorphism between $Ker(D_{i})$ and $Ker(L_{n-i-1}^*)$. In the same way by the first point of the proposition we know that $\phi_{i}^{*}$ induces an isomorphism between $Ker(L_{n-i})$ and $Ker(D_{i-1}^*)$.  But now we know that $\phi_{i}^{*}=\phi_{i}^{-1}$ and so we can conclude that for each $i$ $\phi_{i}$ induces an isomorphism between $Ker(D_{i})\cap Ker(D_{i-1}^*)$ and $Ker(L_{n-i})\cap Ker(L_{n-i-1})^{*}$, that is an isomorphism  between $\mathcal{H}^{i}(H_{*},D_{*})$ and $\mathcal{H}^{n-i}(H_{*},L_{*})$.
\end{proof}

\begin{prop}
Let $(H_{i},D_{i}),\  i=0,...,n$ be a Hilbert complex and suppose that for each $i$ there exists $\phi_{i}:H_{i}\rightarrow H_{n-i}$ that is linear, continuous and bijective. Then there exists a Hilbert complex $(H_{i},L_{i})$ such that the complexes $(H_{i},D_{i})$ and $ (H_{i},L_{i})$  are related  with $\{\phi_{i}\}$ as duality maps. Moreover if each $\phi_{j}$ is an isometry then  the complexes $(H_{i},D_{i})$ and $ (H_{i},L_{i})$  are complementary  with $\{\phi_{i}\}$ as duality maps.
\end{prop}

\begin{proof} Consider the following complexes $(H_{i},L_{i})$ where each $L_{i}$ is the adjoint of the closed and densely defined operator $(\phi_{n-i}\circ D_{n-i-1}\circ \phi_{n-i-1}^{-1}):H_{i+1}\rightarrow H_{i}.$ It clear that $(H_{i},L_{i})$ is a Hilbert complex and by its construction it follows immediately that $(H_{i},D_{i})$ and $(H_{i},L_{i})$ are a pair of related Hilbert complexes having the maps $\{\phi_{i}\}$ as duality maps. Finally it is clear that if each $\phi_{j}$ is an isometry then  the complexes $(H_{i},D_{i})$ and $ (H_{i},L_{i})$  are complementary  with $\{\phi_{i}\}$ as link maps.
\end{proof}

Now we give the following definition which we will use later.

\begin{defi} Let $V_{0}, V_{1},...,V_{n}$ be  a finite sequence of finite dimensional  vector spaces. We will say that it is a finite sequence of finite dimensional  vector spaces with Poincar\'e duality if for each $i$: $$V_{i}\cong V_{n-i}$$ that is $V_i$ and $V_{n-i}$ are isomorphic.
\label{paolo} 
\end{defi}

We are now in position to state the first of the two  main results of this section.

\begin{teo}
\label{duality}
Let $(H_{j},D_{j})\subseteq (H_{j},L_{j})$ be  a pair of complementary Hilbert complexes. Let $i_{r,j}^{*}$ be the map defined in \eqref{oppio}.
Suppose that for each $j$ 
\begin{equation}
\im(\overline{H}^{j}(H_{*},D_{*})\stackrel{i^{*}_{r,j}}{\longrightarrow}\overline{H}^{j}(H_{*},L_{*}))
\end{equation}
is finite dimensional. Then 
\begin{equation}
\im(\overline{H}^{j}(H_{*},D_{*})\stackrel{i^{*}_{r,j}}{\longrightarrow}\overline{H}^{j}(H_{*},L_{*})),\  j=0,...,n
\label{kiol}
\end{equation}
is a finite sequence of finite dimensional  vector spaces with Poincar\'e duality.
\end{teo}

Now we prove some propositions which we will use in the proof of Theorem \ref{duality}.

\begin{prop}
\label{cullen}
Let $H,K$ be two Hilbert spaces and let $T:H\rightarrow K$ be a linear and continuous map. Let $T^{*}:K\rightarrow L$ be the adjoint of $T$. Suppose that $ran(T)$ is closed. Then $$T:Ker(T)^{\bot}\longrightarrow Ker(T^{*})^{\bot}$$ is continuous, bijective with bounded inverse.
\end{prop}

\begin{proof}
We have $K=Ker(T^{*})\oplus Ker(T^{*})^{\bot}$ and $Ker(T^{*})^{\bot}=\overline{ran(T)}$. Therefore by the fact that $ran(T)$ is closed it follows that $T$ is a bijection between $Ker(T)^{\bot}$ and $Ker(T^{*})^{\bot}$. Now from the fact that $Ker(T)^{\bot}$ and $(Ker(T^{*}))^{\bot}$ are closed subspace of $H$ and $K$ respectively it follows we can look at them as Hilbert spaces with the  products  induced by the products of $H$ and $K$ respectively. In this way we can use the closed graph theorem to conclude  that $T|_{Ker(T)^{\bot}}$ has a continuous inverse. 
\end{proof}

\begin{prop}
\label{qwwq}
Let $H$ be a Hilbert space and let $M,N$ be  two closed subspaces of it.  Let $\pi_{M},\pi_{N}$  be the orthogonal projections on $M$ and $N$ respectively. Consider $M$ and $N$ as Hilbert spaces with the scalar product induced by the one of $H$. Then $$\pi_{M}|_{N}=(\pi_{N}|_{M})^{*}$$ that is if we look at $\pi_{M}|_{N}$ as a linear and continuous map from the Hilbert space $M$ to the Hilbert space $N$ then $\pi_{N}|_{M}$ is its adjoint.
\end{prop}

\begin{proof} During the proof we use $<,>_{H}$ to indicate the scalar product of $H$ and $<,>_{M},<,>_{N}$ to indicate the scalar products induced by $<,>_{H}$ on $M$ and $N$ respectively. For each $u\in M,\ v\in N$ we have $<\pi_{N}(u),v>_{N}=<\pi_{N}(u)+\pi_{N^{\bot}}(u),v>_{H}=<u,v>_{H}=<u,\pi_{M}(v)+\pi_{M^{\bot}}(v)>_{H}=<u,\pi_{M}>_{M}$ and so we get the assertion.
\end{proof}
 
Now we are in position to prove Theorem \ref{duality} .
\begin{proof}
First of all, for the benefit of the reader, we explain the strategy of the proof. The main idea is to build a family of maps, that we will label with $\pi_{1,j}:\mathcal{H}^j(H^*,D_*)\rightarrow \mathcal{H}^j(H^*,L_*)$ $j=0,...,n,$ such that:
\begin{itemize}
\item $\pi_{1,j}(\mathcal{H}^{j}(H_{*},D_{*}))\cong  \im(\overline{H}^{j}(H_{*},D_{*})\stackrel{i^{*}_{r,j}}{\longrightarrow}\overline{H}^{j}(H_{*},L_{*}))$ for each $j$
\item $\pi_{1,j}\circ (\phi_j)^{-1}$ is an isomorphism between $\pi_{1,n-j}(\mathcal{H}^{n-j}(H^*,D_*))$ and $\pi_{1,j}(\mathcal{H}^{j}(H^*,D_*))$
\end{itemize}
Therefore, taking the composition of these  maps, we will get the desired isomorphism: $$ \im(\overline{H}^{j}(H_{*},D_{*})\stackrel{i^{*}_{r,j}}{\longrightarrow}\overline{H}^{j}(H_{*},L_{*}))\cong \im(\overline{H}^{n-j}(H_{*},D_{*})\stackrel{i^{*}_{r,n-j}}{\longrightarrow}\overline{H}^{n-j}(H_{*},L_{*})).$$

Now we start defining $\pi_{1,j}$. From Proposition \ref{kio} we know that $$H_{j}=\mathcal{H}^{j}(H_{*},D_{*})\bigoplus\overline{ran(D_{j-1})}\bigoplus\overline{ran(D_{j}^{*})}$$  and that $$H_{j}=\mathcal{H}^{j}(H_{*},L_{*})\bigoplus\overline{ran(L_{j-1})}\bigoplus\overline{ran(L_{j}^{*})}.$$ So for each $j$ we can define $\pi_{D_{j}}$ as the orthogonal projection of $H_{j}$ on $\mathcal{H}^{j}(H_{*},D_{*})$ and  $\pi_{L_{j}}$ as the orthogonal projection of $H_{j}$ on $\mathcal{H}^{j}(H_{*},L_{*})$. In the same way we can define $\pi_{\overline{ran(D_{j-1})}}$, $\pi_{\overline{ran(L_{j-1})}}$, $\pi_{\overline{ran(D_{j}^{*})}}$ and $\pi_{\overline{ran(L_{j}^{*})}}$. Finally we define$$ \pi_{1,j}:=(\pi_{L_{j}})|_{\mathcal{H}^{j}(H_{*},D_{*})},\  \pi_{2,j}:=(\pi_{\overline{ran(L_{j-1})}})|_{\mathcal{H}^{j}(H_{*},D_{*})},\  \pi_{3,j}:=(\pi_{\overline{ran(L_{j}^{*})}})|_{\mathcal{H}^{j}(H_{*},D_{*})}.$$ Analogously,  but now projecting from $\mathcal{H}^{j}(H_{*},L_{*})$ on the orthogonal components of the sum $H_{j}=\mathcal{H}^{j}(H_{*},D_{*})\bigoplus\overline{ran(D_{j-1})}\bigoplus\overline{ran(D_{j}^{*})}$,  we define $\pi_{4,j}:\mathcal{H}^j(H^*,L_*)\longrightarrow \mathcal{H}^j(H^*,D_*),\\ \pi_{5,j}:\mathcal{H}^j(H^*,L_*)\longrightarrow \overline{ran(D_{j-1})}$ and  $\pi_{6,j}:\mathcal{H}^j(H^*,L_*)\longrightarrow \overline{ran(D_{j}^*)}.$\\Our  claim now is to show that for each $j$
\begin{equation}
\pi_{1,j}(\mathcal{H}^{j}(H_{*},D_{*}))\cong  \im(\overline{H}^{j}(H_{*},D_{*})\stackrel{i^{*}_{r,j}}{\longrightarrow}\overline{H}^{j}(H_{*},L_{*}))
\label{allaz}
\end{equation}
Let $[h]\in \overline{H}^{j}(H_{*},D_{*})$ be a cohomology class. By \eqref{pppp} we know that there exists a unique representative of $[h]$ in $\mathcal{H}^{j}(H_{*},D_{*})$. We call it $\omega$. Every other representative of $[h]$ differs from $\omega$ by an element in $\overline{ran(D_{j-1})}$; therefore $i^{*}_{r,j}([h])=[i_{j}(\omega)]$. Now we can decompose $\omega$ as $\omega=\pi_{1,j}(\omega)+\pi_{2,j}(\omega)+\pi_{3,j}(\omega)$. Clearly $[i_j(\omega)]=[\pi_{1,j}(\omega)]+[\pi_{3,j}(\omega)]$.  So if we show that $\pi_{3,j}|_{\mathcal{H}^{j}(H_{*},D_{*})}\equiv 0$ we get the claim. Now let $\eta\in \mathcal{H}^{j}(H_{*},D_{*}) .$  Then $\pi_{3,j}(\eta)\in \overline{ran(L_{j}^{*}})\cap Ker(L_{j})$  because $\pi_{3,j}(\eta)=\eta-\pi_{1,j}(\eta)-\pi_{2,j}(\eta)$ and each term on the right hand side of the equality lies in $Ker(L_{j})$. But $(Ker(L_{j}))^{\bot}=\overline{ran(L_{j}^{*})}$ and therefore $\pi_{3,j}(\eta)=0$. So for each 
$\eta\in \mathcal{H}^{j}(H_{*},D_{*})$ we have $\pi_{3,j}(\eta)=0$. Therefore the  claim is proved.\\In this way the first  point in the sketch of the strategy described above is proved. Now, in order to complete the proof,  we have to prove the second one.
First of all we observe that now we know that $\pi_{1,j}$ has closed range because it is isomorphic to $\im(\overline{H}^{j}(H_{*},D_{*})\stackrel{i^{*}_{r,j}}{\longrightarrow}\overline{H}^{j}(H_{*},L_{*}))$ which is finite dimensional by the assumptions. Moreover we know  that $Ker(\pi_{1,j})=\overline{ran(L_{j-1})}\cap \mathcal{H}^{j}(H_{*},D_{*}).$ In the same way we can prove   that $Ker(\pi_{4,j})=\overline{ran(D_{j}^*)}\cap \mathcal{H}^{j}(H_{*},L_{*})$. Finally from the observations above and from Propositions \ref{cullen} and    \ref{qwwq} we get the following three properties for each $j$:
\begin{enumerate}
\item $(\pi_{1,j})^{*}=\pi_{4,j}$ and both induce an isomorphism between $ran(\pi_{4,j})$ and $ran(\pi_{1,j})$.
\item $ \mathcal{H}^{j}(H_{*},D_{*})=ran(\pi_{4,j})\oplus( \overline{ran(L_{j-1})}\cap \mathcal{H}^{j}(H_{*},D_{*}))=ran(\pi_{4,j})\oplus Ker(\pi_{1,j})$.
\item $ \mathcal{H}^{j}(H_{*},L_{*})=ran(\pi_{1,j})\oplus( \overline{ran(D_{j}^*)}\cap \mathcal{H}^{j}(H_{*},L_{*}))=ran(\pi_{1,j})\oplus Ker(\pi_{4,j})$.
\end{enumerate}
By the fourth point of Proposition \ref{errew}  we know that each $\phi_{j}$ induces an isomorphism between $\mathcal{H}^{j}(H_{*},D_{*})$ and $\mathcal{H}^{n-j}(H_{*},L_{*})$ . For the same reason $\phi_{j}$  induces an isomorphism between  $\overline{ran(L_{j-1})}$ and $\overline{ran(D_{n-j}^*)}$ and between $\overline{ran(D_{j-1})}$ and $\overline{ran(L_{n-j}^*)}$  . This implies that each $\phi_{j}$ induces an isomorphism between $\mathcal{H}^{j}(H_{*},D_{*})\cap \overline{ ran(L_{j-1})}$ and $\mathcal{H}^{n-j}(H_{*},L_{*})\cap\overline{ ran(D_{n-j}^*)}$ that is an isomorphism between $Ker(\pi_{1,j})$ and $Ker(\pi_{4,n-j})$. In this way we can conclude that each $\phi_{j}$ induces an isomorphism between $$\frac{\mathcal{H}^{j}(H_{*},D_{*})}{Ker(\pi_{1,j})}\  \text{ and}\  \frac{\mathcal{H}^{n-j}(H_{*},L_{*})}{Ker(\pi_{4,n-j})}.$$But, as recalled above,  $(\pi_{1,j})^{*}=\pi_{4,j}$, they have both closed range  and they both induce an isomorphism between $ran(\pi_{4,j})$ and $ran(\pi_{1,j})$. Therefore we get: $$\frac{\mathcal{H}^{j}(H_{*},D_{*})}{Ker(\pi_{1,j})}\cong ran(\pi_{4,j})\cong ran(\pi_{1,j})\cong \im(\overline{H}^{j}(H_{*},D_{*})\stackrel{i^{*}_{r,j}}{\longrightarrow}\overline{H}^{j}(H_{*},L_{*}))$$ and similarly $$\frac{\mathcal{H}^{n-j}(H_{*},L_{*})}{Ker(\pi_{4,n-j})}\cong ran( \pi_{1,n-j})\cong  \im(\overline{H}^{n-j}(H_{*},D_{*})\stackrel{i^{*}_{r,n-j}}{\longrightarrow}\overline{H}^{n-j}(H_{*},L_{*})) .$$ The composition of the above isomorphisms gives: $$ \im(\overline{H}^{j}(H_{*},D_{*})\stackrel{i^{*}_{r,j}}{\longrightarrow}\overline{H}^{j}(H_{*},L_{*}))\cong  \im(\overline{H}^{n-j}(H_{*},D_{*})\stackrel{i^{*}_{r,n-j}}{\longrightarrow}\overline{H}^{n-j}(H_{*},L_{*}))$$ and this completes the proof.
\end{proof}

\begin{rem}
\label{piop}
 By the above proof we get that given a pair of Hilbert complexes $(H_{*},D_{*})\subseteq (H_{*},L_{*})$, without any other assumption, the following isomorphism holds for each $j$ :
\begin{equation}
\label{planterwald}
ran(\pi_{1,j})\cong  \im(\overline{H}^{j}(H_{*},D_{*})\stackrel{i^{*}_{r,j}}{\longrightarrow}\overline{H}^{j}(H_{*},L_{*})). 
\end{equation}
Moreover when the sequences of vector spaces on the right hand side  of \eqref{planterwald}   is finite dimensional we have 
 $$ \mathcal{H}^{j}(H_{*},D_{*}) \cap(\mathcal{H}^{j}(H_{*},D_{*})\cap\overline{ran(L_{j-1})})^{\bot}\cong  (\mathcal{H}^{j}(H_{*},L_{*})\cap\overline{ran(D_{j}^*)})^{\bot}\cap\mathcal{H}^{j}(H_{*},L_{*}) $$that is $$ran(\pi_{1,j})\cong ran(\pi_{4,j}).$$
\end{rem}

The following  statements are immediate consequences of Theorem \ref{duality}.
\begin{cor}
Suppose that one of the two complexes of Theorem \ref{duality} is Fredholm; then also the other complex is Fredholm and 
\begin{equation}
\im(H^{j}(H_{*},D_{*}) \longrightarrow H^{j}(H_{*},L_{*})),\ j=0,...,n
\end{equation}
is a finite sequence of finite dimensional  vector spaces with Poincar\'e duality.  Moreover 
\begin{equation}
ran(\pi_{1,j})\cong \im(H^{j}(H_{*},D_{*}) \longrightarrow H^{j}(H_{*},L_{*})).
\end{equation}
and 
\begin{equation}
\mathcal{H}^{j}(H_{*},D_{*})\cap (\mathcal{H}^{j}(H_{*},D_{*})\cap ran(L_{j-1}))^{\bot}\cong  (\mathcal{H}^{j}(H_{*},L_{*})\cap ran(D_{j}^*))^{\bot}\cap\mathcal{H}^{j}(H_{*},L_{*}) .
\end{equation}
\label{klop}
\end{cor}

\begin{prop}
Let $(H_{*},D_{*})\subseteq (H_{*},L_{*})$ be a pair of complementary Hilbert complexes. Furthermore suppose that there is a third Hilbert complex  $(H_{*},P_{*})$ with the following properties:
\begin{enumerate}
\item  $(H_{*},D_{*})\subseteq (H_{*},P_{*})\subseteq (H_{*},L_{*})$.
\item The reduced cohomology of  $(H_{*},P_{*})$  is finite dimensional.
\end{enumerate} 
Then $$\im(\overline{H}^{j}(H_{*},D_{*})\stackrel{i^{*}_{r,j}}{\longrightarrow}\overline{H}^{j}(H_{*},L_{*})),\  j=0,...,n
$$ is a finite  sequence of finite dimensional vector spaces with Poincar\'e duality.
\end{prop}

\begin{proof}
The assertion is an immediate consequence of the following, simple fact. Let $i_{1,j}$ be the natural inclusion of $(H_{*},D_{*})$ in $(H_{*},P_{*})$, let $i_{2,j}$ be the natural inclusion of $(H_{*},P_{*})$ in $(H_{*},L_{*})$ and finally let  $i_{3,j}$ be the natural inclusion of $(H_{*},D_{*})$ in $(H_{*},L_{*})$. Obviously we have $i_{3,j}=i_{2,j}\circ i_{1,j}$. This implies that also the respective maps induced between the reduced cohomology groups commute. So we have $i_{r,3,j}^*=i_{r,2,j}^*\circ i_{r,1,j}^*$ and therefore $$\im(\overline{H}^{j}(H_{*},D_{*})\stackrel{i^{*}_{r,3,j}}{\longrightarrow}\overline{H}^{j}(H_{*},L_{*}))\subseteq \im(\overline{H}^{j}(H_{*},P_{*})\stackrel{i^{*}_{r,2,j}}{\longrightarrow}\overline{H}^{j}(H_{*},L_{*})).$$ In this way, by the second hypothesis, we know that $$ \im(\overline{H}^{j}(H_{*},D_{*})\stackrel{i^{*}_{r,3,j}}{\longrightarrow}\overline{H}^{j}(H_{*},L_{*}))$$ is a finite sequence of finite dimensional vector spaces. Now we are in position to apply Theorem \ref{duality} and so the proposition follows.
\end{proof}

Finally we conclude this section with the following result:

\begin{teo}
\label{cheafa}
Let  $(H_{i},D_{i})\subseteq (H_{i},L_{i})$,  $i=0,...,n$, be a pair of  Hilbert complexes. Suppose that for each $i$ $ran(D_{i})$ is closed in $H_{i+1}$. Then there exists a third Hilbert complex $(H_{i},P_{i})$ such that: 
\begin{enumerate}
\item  $(H_{i},D_{i})\subseteq (H_{i},P_{i})\subseteq (H_{i},L_{i})$. 
\item $H^i(H_{*},P_*)=\im(H^i(H_*,D_*) \rightarrow H^i(H_*,L_*))$. 
\end{enumerate}
Moreover if $(H_{i},D_{i})\subseteq (H_{i},L_{i})$ are complementary and  $(H_{i},D_{i})$, or equivalently $(H_{i},L_{i})$, is Fredholm then $(H_{i},P_{i})$ is a  Fredholm complex with Poincar\'e duality.
\end{teo}

\begin{proof}
It is immediate that $$\im(H^i(H_*,D_*) \rightarrow H^i(H_*,L_*))=\frac{Ker(D_i)}{ran(L_{i-1})\cap \mathcal{D}(D_i)}.$$ Therefore for each $i=0,...,n$ we have to construct a closed extension of $D_i$, that we call $P_i$, such that:  
\begin{equation}
\label{sudsudan}
Ker(P_{i})= Ker(D_i)\ \text{and}\ ran(P_{i-1})=ran(L_{i-1})\cap \mathcal{D}(D_i).
\end{equation}
In order to get this closed extensions $P_i$, we have to build a suitable subspace of $\mathcal{D}(L_i)$, let us say $B_i$, such that:
\begin{itemize}
\item  $\mathcal{D}(D_i)\subset B_i\subset \mathcal{D}(L_i)$.
\item $P_i:H_i\rightarrow H_{i+1}$ with domain given by $B_i$ and defined as the restriction of $L_i$ to $B_i$ is a closed operator.
\item \eqref{sudsudan} holds.
\end{itemize}
 To do this, from now on we will consider the following Hilbert space $(\mathcal{D}(L_i),<\ , >_{\mathcal{G}})$, which is by definition  the domain of $L_i$ endowed with the graph scalar product. Therefore all the direct sums that will appear  and all the assertions of topological type are referred to this Hilbert space $(\mathcal{D}(L_i),<\ , >_{\mathcal{G}})$. We can decompose $(\mathcal{D}(L_i),<\ , >_{\mathcal{G}})$ in the following way:
\begin{equation}
\label{wsws}
(\mathcal{D}(L_i),<\ , >_{\mathcal{G}})=Ker(L_i)\oplus V_i
\end{equation}
where $V_i=\{\alpha\in \mathcal{D}(L_i)\cap \overline{ran(L_{i}^*)}\}$. They are both closed  in $(\mathcal{D}(L_i),<\ , >_{\mathcal{G}})$ because $V_i$ is  the orthogonal complement of $Ker(L_i)$ in $(\mathcal{D}(L_i),<\ ,>_{\mathcal{G}})$ and $Ker(L_i)$ is closed because $L_i:(\mathcal{D}(L_i),<\ , >_{\mathcal{G}})\rightarrow H_{i+1}$ is continuous.\\Consider now $(\mathcal{D}(D_i),<\ , >_{\mathcal{G}})$. By the fact that $D_i$ is a closed operator we get that $\mathcal{D}(D_i)$ is a closed subspace of $(\mathcal{D}(L_i),<\ , >_{\mathcal{G}})$. Moreover  we can decompose $\mathcal{D}(D_i)$ as: 
\begin{equation}
\label{wswsa}
(\mathcal{D}(D_i),<\ ,>_{\mathcal{G}})=Ker(D_i)\oplus A_i. 
\end{equation}
Analogously to the previous case  $A_i=\{\alpha\in \mathcal{D}(D_i)\cap \overline{ran(D_{i}^*)}\}$. Furthermore they are both closed  in $(\mathcal{D}(D_i),<\ ,>_{\mathcal{G}})$ because, in a similar way to \eqref{sudsudan},  $A_i$ is  the orthogonal complement of $Ker(D_i)$ in $(\mathcal{D}(D_i),<\ ,>_{\mathcal{G}})$ and $Ker(D_i)$ is closed because $D_i:(\mathcal{D}(D_i),<\ , >_{\mathcal{G}})\rightarrow H_{i+1}$ is continuous. Now let $C_i=\{\alpha \in \mathcal{D}(L_i): L_i(\alpha)\in \mathcal{D}(D_{i+1})\}$. $C_i$ is closed in $(\mathcal{D}(D_i),<\ ,>_{\mathcal{G}})$ because it is the preimage of a closed subspace under a continuous map, that is $C_i=L_i^{-1}(Ker(L_{i+1}))$. Finally let $$W_i:=C_i\cap V_i.$$ Then it is clear that: 
\begin{equation}
\label{wswsas}
C_i=Ker(L_i)\oplus W_i. 
\end{equation}
Obviously if $Ker(D_i)=Ker(L_i)$ then it enough to define $P_i:=L_i|_{C_i}$. So we can suppose that $Ker(D_i)$ is properly contained in $Ke r(L_i)$. Let $\pi_1$ be the orthogonal projection of $A_i$ onto $Ker(L_i)$ and analogously let $\pi_2$ be the orthogonal projection of $A_i$ onto $V_i$. We have the following properties:
\begin{enumerate}
\item $\pi_2$ is injective 
\item $ran(\pi_2)\subseteq W_i$
\item $ran(\pi_2)$ is closed.
\end{enumerate}
The first property follows from the fact that  $Ker(\pi_2)=A_i\cap Ker(L_i)$. But $L_i$ is an extension of $D_i$; therefore if an element $\alpha$ lies in  $A_i\cap Ker(L_i)$ then it lies also in $Ker(D_i)$. So we can say that $\alpha\in Ker(D_i)\cap A_i$ and this, combined with \eqref{wswsa}, implies that $\alpha=0$. For the second property, given $\alpha\in A_i$, we have $D_i(\alpha)=L_i(\alpha)=L_i(\pi_1(\alpha)+\pi_2(\alpha))=L_i(\pi_2(\alpha))$ and therefore $\pi_2(\alpha)\in W_i$. Finally, for  the third property, consider a sequence $\{\gamma_{m}\}_{m\in \mathbb{N}}\subset A_i$ such that $\pi_{2}(\gamma_{m})$ converges to $\gamma\in W_i$. We recall that we are using $(\mathcal{D}(L_i),<\ ,>_{\mathcal{G}})$ and therefore saying that $\pi_2(\gamma_m)$ converges to $\gamma$ means that $\pi_2(\gamma_m)$ converges to $\gamma$ in $H_i$ and $L_i(\pi_2(\gamma_m))$ converges to $L_i(\gamma)$ in $H_{i+1}.$ Then: $$\lim_{m\rightarrow \infty}D_{i}(\gamma_m)=\lim_{m\rightarrow \infty}L_{i}(\gamma_m)=\lim_{m\rightarrow \infty}L_{i}(\pi_{2}(\gamma_m))=L_i(\gamma).$$ This implies that $$\lim_{m\rightarrow \infty}D_{i}(\gamma_m)=L_i(\gamma)$$ and therefore the limit  exists.  So by the assumptions about the range of $D_i$ we get that there exists an element $\eta\in A_i$ such that $$\lim_{m\rightarrow \infty}D_{i}(\gamma_m)=D_{i}(\eta).$$Moreover $L_{i}(\gamma)=D_{i}(\eta)=L_{i}(\eta)=L_{i}(\pi_{2}(\eta))$. This implies that $L_i(\pi_{2}(\eta)-\gamma)=0$ and therefore $\pi_2(\eta)=\gamma$ because $\pi_2(\eta),\gamma \in W_i$ and $L_i$ is injective on $W_i$. In this way we showed that $\pi_2$ is closed.\\
 Now define $N_i$ as the orthogonal complement of $ran(\pi_2)$ in $W_i$. Then for each $\alpha\in A_i$ and for each $\beta\in N_i$ we have $<\alpha,\beta>_{\mathcal{G}}=<\pi_1(\alpha)+\pi_2(\alpha),\beta>_{\mathcal{G}}=0$. This last property, joined with the fact that both $A_i$ and $N_i$ are closed subspaces of $(\mathcal{D}(L_i),<\ ,>_{\mathcal{G}})$,  implies that the vector space generated by $A_i$ and $N_i$ is closed and, if we call it $M_i$, then $M_i$ satisfies the following orthogonal decomposition: $M_i=A_i\oplus N_i$. Again for each $\alpha\in Ker(D_i)$ and for each $\beta \in M_i$ we have $<\alpha,\beta>_{\mathcal{G}}=0.$ This is because for each $\beta \in M_i$ there exist  unique $\beta_1\in A_i,\ \beta_2\in N_i$ such that $\beta=\beta_1\oplus \beta_2$. Now it is clear that $<\alpha,\beta_1>_{\mathcal{G}}=0=<\alpha,\beta_2>_{\mathcal{G}}$ because $Ker(D_i)\subset Ker(L_i)$, $N_i\subset W_i$, $W_i$ and $Ker(L_i)$ are orthogonal and $Ker(D_i)$ and $A_i$ are orthogonal. Therefore, also in this case, if we call $B_i$ the vector space generated by $Ker(D_i)$ and $M_i$ we have that 
\begin{equation}
\label{piedeghiaccio}
B_i=Ker(D_i)\oplus M_i=Ker(D_i)\oplus A_i\oplus N_i=\mathcal{D}(D_i)\oplus N_i
\end{equation}
 and therefore $B_i$ is a closed subspace of $(\mathcal{D}(L_i),<\ ,>_{\mathcal{G}})$. Finally define $P_i$ as 
\begin{equation}
\label{toposorce}
P_i:=L_i|_{B_i}
\end{equation}  
By the construction it is clear that for each $\alpha\in B_i$ we have $P_i(\alpha)\in \mathcal{D}(D_{i+1})\cap ran(L_i)$ and that  $\mathcal{D}(D_i)\subset B_i$. Therefore this implies that the composition $P_{i+1}\circ P_i$ is defined on the whole $B_i$ and that $P_{i+1}\circ P_i\equiv 0.$ Moreover, if we look at $P_i$ as an unbounded operator from $H_i$ to $H_{i+1}$, then it is clear that $P_i$ is densely defined because $\mathcal{D}(D_i)\subset B_i$ and $\mathcal{D}(D_i)$ is dense in $H_i$. Moreover it  is also easy to see that $P_i$ is a closed operator because it is defined as the restriction of $L_i$, which is a closed operator, on a closed subspace of $(\mathcal{D}(L_i), <\ ,>_{\mathcal{G}})$.\\
To conclude the proof we have to check that $Ker(P_i)=Ker(D_i)$ and that $ran(P_i)=ran(L_i)\cap \mathcal{D}(D_{i+1})$.
Let $\alpha\in Ker(P_i)$. According to \eqref{piedeghiaccio} we can decompose $\alpha$ in a unique way as \begin{equation}
\label{nonfiniscemai}
\alpha=\alpha_1+\alpha_2+\alpha_3
\end{equation}
where $\alpha_1\in Ker(D_i)$, $\alpha_2\in A_i$ and $\alpha_3\in N_i$. The goal now is to show that $0=\alpha_2=\alpha_3$. The assumption on $\alpha$ implies that $\alpha_2+\alpha_3\in Ker(P_i)$ because $\alpha\in Ker(P_i)$ and $\alpha_1\in Ker(D_i)$. We can decompose $\alpha_2$ in a unique way  as $\alpha_2=\beta_1+\beta_2$ where $\beta_1\in ran(\pi_1)$ and $\beta_2\in ran(\pi_2)$. Therefore  we obtain that $L_i(\beta_2+\alpha_3)=0$ because $\beta_1+\beta_2+\alpha_3=\alpha_2+\alpha_3\in  Ker(L_i)$ and $\beta_1\in Ker(L_i)$. This implies that $\beta_2+\alpha_3\in W_i\cap Ker(L_i)$ and therefore from \eqref{wswsas} we can conclude that $\beta_2+\alpha_3=0$. But $\beta_2+\alpha_3\in ran(\pi_2)\oplus N_i,\ \beta_{2}\in ran(\pi_2)$, $\alpha_3\in N_i$ and so we get $0=\beta_2=\alpha_3$. This in turn implies that $\alpha_2=\beta_1$ that is $\alpha_2\in A_i\cap Ker(L_i)=Ker(\pi_2)$. By the injectivity of $\pi_2$ previously proved, we get that $\alpha_2=0$ and therefore \eqref{nonfiniscemai} becomes $\alpha=\alpha_1\in Ker(D_i)$. So we got $Ker(P_i)\subseteq Ker(D_i)$; the other inclusion is trivial and therefore we have $Ker(P_i)= Ker(D_i)$.\\ Now we have to check that $ran(P_i)=ran(L_i)\cap \mathcal{D}(D_{i+1})$. Clearly, as observed above, the inclusion $\subseteq$ follows immediately by the construction of $P_i$. So we have to prove the converse. Let $\psi\in ran(L_i)\cap \mathcal{D}(D_{i+1})$. Then there exists a unique  element $\gamma\in W_i$ such that $L_i(\gamma)=\psi$.
Moreover there exist and are unique $\gamma_1\in ran(\pi_2)$ and $\gamma_2\in N_i$ such that $\gamma=\gamma_1+\gamma_2$. Now let $\theta\in A_i$ be the unique element in $A_i$ such that $\pi_2(\theta)=\gamma_1$. Finally consider $\theta+\gamma_2$. By construction $\theta+\gamma_2\in B_i$ and $P_i(\theta+\gamma_2)=L_i(\theta+\gamma_2)=L_i(\pi_1(\theta)+\pi_2(\theta)+\gamma_2)=L_i(\gamma_1+\gamma_2)=L_i(\gamma)$.
In this way we showed that $ran(L_i)\cap \mathcal{D}(D_{i+1})=ran(P_i)$. \\Finally  suppose that  $(H_{i},D_{i})$ and  $(H_{i},L_{i})$ are complementary and that $(H_{i},D_{i})$, or equivalently  $(H_{i},L_{i})$, is Fredholm. We have the following  natural and surjective map: 
\begin{equation}
\label{pappapappa}
\frac{Ker(D_{i+1})}{ran(D_i)}\longrightarrow \frac{Ker(D_{i+1})}{ran(P_i)}.
\end{equation}
By the assumptions  $H^i(H_*,D_*)$ is finite dimensional and this, using \eqref{pappapappa},  implies that also $H^i(H_*,P_*)$ is finite dimensional, that is $(H_i,P_i)$ is a Fredholm complex. Now using Theorem \ref{duality} it follows that Poincar\'e duality holds for $(H_{i},P_{i})$. This completes the proof. 
\end{proof}

\section{Geometric Applications}

\subsection{Duality for reduced $L^{2}-$cohomology}

Now we want to show that riemannian geometry is a context in which  pairs  of complementary Hilbert complexes appear in a natural way.\\
Let  $(M,g)$ be  an open and oriented riemannian manifold of dimension $m$ and let $E_{0},...,E_{n}$ be vector bundles over $M$. For each $i=0,...,n$ let $C^{\infty}_{c}(M,E_{i})$ be the space of smooth section with compact support. If we put on each vector bundle a metric $h_{i}\ i=0,...,n$ then we can construct in a natural way a sequences of Hilbert space $L^{2}(M,E_{i}),\ i=0,...,n$ as the completion of $C^{\infty}_{c}(M,E_{i}).$ Now suppose that we have a complex of differential operators :
\begin{equation}
0\rightarrow C^{\infty}_{c}(M,E_{0})\stackrel{P_{0}}{\rightarrow}C^{\infty}_{c}(M,E_{1})\stackrel{P_{1}}{\rightarrow}C^{\infty}_{c}(M,E_{2})\stackrel{P_{2}}{\rightarrow}...\stackrel{P_{n-1}}{\rightarrow}C^{\infty}_{c}(M,E_{n})\rightarrow 0,
\label{yygg}
\end{equation}
To turn this complex into a Hilbert complex we must specify a closed extension of $P_{*}$ that is an operator between $L^{2}(M,E_{*})$ and  $L^{2}(M,E_{*+1})$ with closed graph which is an extension of $P_{*}$.  So we state some definitions and propositions which generalize those  stated in \cite{FB}. We start recalling  the two canonical closed extensions of $P$.

\begin{defi} The maximal extension $P_{max}$; this is the operator acting on the domain:
\begin{equation} 
\mathcal{D}(P_{max,i})=\{\omega\in L^{2}(M,E_{i}): \exists\ \eta\in L^{2}(M,E_{i+1})
\end{equation}

$$ s.t.\ <\omega,P^t_{i}\zeta>_{L^{2}(M,E_{i})}=<\eta,\zeta>_{L^{2}(M,E_{i+1})}\ \forall\ \zeta\in C_{0}^{\infty}(M,E_{i+1}) \}$$ where $P^t_i$ is the formal adjoint of $P_i$.

In this case $P_{max,i}\omega=\eta.$ In  other words $\mathcal{D}(P_{max,i})$ is the largest set of forms $\omega\in L^{2}(M, E_{i})$ such that $P_{i}\omega$, computed distributionally, is also in $L^{2}(M,E_{i+1}).$
\label{nino}
\end{defi}

\begin{defi} The minimal extension $P_{min,i}$; this is given by the graph closure of $P_{i}$ on $C_{0}^{\infty}(M, E_{i})$ with respect to the norm of $L^{2}(M,E_{i})$, that is,
\begin{equation} \mathcal{D}(P_{min,i})=\{\omega\in L^{2}(M,E_{i}): \exists\ \{\omega_{j}\}_{j\in J}\subset C^{\infty}_{0}(M,E_{i}),\ \omega_{j}\rightarrow \omega ,\       P_{i}\omega_{j}\rightarrow \eta\in L^{2}(M,E_{i+1})\} 
\end{equation}
and in this case $P_{min,i}\omega=\eta$
\label{reso}
\end{defi}

Obviously $\mathcal{D}(P_{min,i})\subset \mathcal{D}(P_{max,i})$. Furthermore, from these definitions, it follows immediately that $$P_{min,i}(\mathcal{D}(P_{min,i}))\subset \mathcal{D}(P_{min,i+1}),\ P_{min,i+1}\circ P_{min,i}=0$$ and that $$P_{max,i}(\mathcal{D}(P_{max,i}))\subset \mathcal{D}(P_{max,i+1}),\ P_{max,i+1}\circ P_{max,i}=0.$$ \\Therefore $(L^{2}(M,E_{*}),P_{max/min,*})$ are both Hilbert complexes and their cohomology groups, reduced cohomology groups, are denoted respectively by $H_{2,max/min}^{i}(M, E_{*})$ and  $\overline{H}_{2,max/min}^{i}(M,E_{*})$. 

Another straightforward but important fact is that the Hilbert complex adjoint of \\$(L^{2}(M,E_{*}),P_{max/min,*})$ is $(L^{2}(M,E_{*}),P^t_{min/max,*})$, that is
\begin{equation}
(P_{max,i})^{*}=P^t_{min,i},\     (P_{min,i})^{*}=P^t_{max,i}.
\end{equation}

Using  Proposition \ref{kio}  we obtain two weak Kodaira decompositions:
\begin{equation}
L^{2}(M,E_{i})=\mathcal{H}^{i}_{abs/rel}(M,E_i)\oplus \overline{ran(P_{max/min,i-1})}\oplus \overline{ran(P^t_{min/max,i})}
\label{kd}
\end{equation}
with summands mutually orthogonal in each case. For the first summand in the right, called the absolute or relative Hodge cohomology, we have by \eqref{said}:
\begin{equation}
\label{nannabobo}
\mathcal{H}^{i}_{abs/rel}(M,E_*)=Ker(P_{max/min,i})\cap Ker(P^t_{min/max,i-1}).
\end{equation}
We can also consider the two natural Laplacians associated to these Hilbert complexes, that is for each $i$
\begin{equation}
\label{kaak}
P_{min,i}^t\circ P_{max,i}+P_{max,i-1}\circ P_{min,i-1}^t
\end{equation}
and 
\begin{equation}
\label{kkll}
P_{max,i}^t\circ P_{min,i}+P_{min,i-1}\circ P_{max,i-1}^t
\end{equation}
with domain described in \eqref{saed}. Using \eqref{said} and \eqref{pppp} it follows that the nullspace of  \eqref{kaak} is isomorphic to the absolute Hodge cohomology which is in turn isomorphic to the reduced  cohomology of the Hilbert complex $(L^{2}(M,E_{*}),P_{max,*})$. Analogously,  using again \eqref{said} and \eqref{pppp}, it follows that the nullspace of  \eqref{kkll} is isomorphic to the relative Hodge cohomology which is in turn isomorphic to the reduced  cohomology of the Hilbert complex $(L^{2}(M,E_{*}),P_{min,*})$.\\
Finally we recall that we can define other two Hodge cohomology groups $\mathcal{H}^i_{max/min}(M,E_*)$ defined as
\begin{equation}
\label{aaaaaaaa}
\mathcal{H}^i_{max/min}(M,E_*)=Ker(P_{max/min,i})\cap Ker(P^t_{max/min,i-1}).
\end{equation}

Now we are in position to state the following results:

\begin{teo}
\label{mzzm}
Let  $(M,g)$ be  an open and oriented riemannian manifold of dimension $m$ and let $E_{0},...,E_{n}$ be vector bundles over $M$ endowed with metrics $h_{i}\ i=0,...,n$. Suppose that we have a complex of differential operators :
\begin{equation}
0\rightarrow C^{\infty}_{c}(M,E_{0})\stackrel{P_{0}}{\rightarrow}C^{\infty}_{c}(M,E_{1})\stackrel{P_{1}}{\rightarrow}C^{\infty}_{c}(M,E_{2})\stackrel{P_{2}}{\rightarrow}...\stackrel{P_{n-1}}{\rightarrow}C^{\infty}_{c}(M,E_{n})\rightarrow 0,
\label{yyggq}
\end{equation}
and let 
\begin{equation}
0\rightarrow L^2(M,E_{0})\stackrel{P_{max,0}}{\rightarrow}L^2(M,E_{1})\stackrel{P_{max,1}}{\rightarrow}L^2(M,E_{2})\stackrel{P_{max,2}}{\rightarrow}...\stackrel{P_{max,n-1}}{\rightarrow}L^2(M,E_{n})\rightarrow 0,
\label{yyggqq}
\end{equation}
and
\begin{equation}
0\rightarrow L^2(M,E_{0})\stackrel{P_{min,0}}{\rightarrow}L^2(M,E_{1})\stackrel{P_{min,1}}{\rightarrow}L^2(M,E_{2})\stackrel{P_{min,2}}{\rightarrow}...\stackrel{P_{min,n-1}}{\rightarrow}L^2(M,E_{n})\rightarrow 0,
\label{yyggqqa}
\end{equation}
the two natural Hilbert complexes associated with \eqref{yyggq} as described above.  Suppose that for each $i=0,...,n$ there exists an isometry $\phi_{i}:(E_{i},h_{i})\rightarrow (E_{n-i},h_{n-i})$; with a little abuse of notation let still $\phi_i$ denote  the induced isometry from $L^2(M,E_i)$ to $L^2(M,E_{n-i})$. Finally suppose  that $P^{t}_{n-i-1}\circ \phi_{i}=c_{i}(\phi_{i+1}\circ P_{i})$, where $c_{i}\neq 0$ is a constant which depends only on  $i$.\\  If $\im(\overline{H}^{i}_{2,min}(M,E_{*})\stackrel{i^{*}_{r,i}}{\longrightarrow}\overline{H}^{i}_{2,max}(M,E_{*}))$ is finite dimensional for each $i$ then $$\im(\overline{H}^{i}_{2,min}(M,E_{*})\stackrel{i^{*}_{r,i}}{\longrightarrow}\overline{H}^{i}_{2,max}(M,E_{*}))$$ is a finite sequence of finite dimensional vector spaces with Poincar\'e duality.
\end{teo}
\begin{proof} From the hypothesis we know that  for each $i=0,...,n$ there exists an isometry $\phi_{i}:(E_{i},h_{i})\rightarrow (E_{n-i},h_{n-i})$ such that $P^{t}_{n-i-1}\circ \phi_{i}=c_{i}(\phi_{i+1}\circ P_{i})$, where $c_{i}\neq 0$ is a constant which depends just on $i$. This isometries of vector bundles induces isometries from $L^{2}(M,E_{i})$ to $L^{2}(M,E_{n-i})$, that with a little abuse of notation we still label $\phi_i$, such that  $\phi_i(\mathcal{D}(P_{min,i}))=\mathcal{D}(P^t_{min,n-i-1})$ and $P^{t}_{min,n-i-1}\circ \phi_{i}=c_{i}(\phi_{i+1}\circ P_{min,i})$. So we showed that the complexes $(L^{2}(M,E_{*}),P_{min,*}) \subseteq (L^{2}(M,E_{*}),P_{max,*})$ are a pair of complementary Hilbert complexes. Now, applying Theorem \ref{duality}, we can get the conclusion.
\end{proof}

\begin{teo}
\label{polca}
In the same hypothesis of the previous theorem, suppose furthermore that for each $i=0,...,n$ $ran(P_{min,i})$ is closed in $L^2(M,E_{i+1})$. Then  there exists a Hilbert complex $(L^2(M,E_i),P_{\mathfrak{m},i})$ such that for each $i=0,...,n$ $$\mathcal{D}(P_{min,i})\subset \mathcal{D}(P_{\mathfrak{m},i})\subset \mathcal{D}(P_{max,i}),$$ $P_{max,i}$ is an extension of $P_{\mathfrak{m},i}$ which is an extension of $P_{min,i}$  and $$H^i_{2,\mathfrak{m}}(M,E_i)=\im(H^{i}_{2,min}(M,E_{*})\stackrel{i^{*}_{i}}{\longrightarrow}H^{i}_{2,max}(M,E_{*}))$$ where $H^i_{2,\mathfrak{m}}(M,E_i)$ is the cohomology of the Hilbert complex $(L^2(M,E_i),P_{\mathfrak{m},i})$. Finally if\\ $(L^2(M,E_i),P_{max,i})$ or equivalently $(L^2(M,E_i),P_{min,i})$ is Fredholm then $(L^2(M,E_i),P_{\mathfrak{m},i})$ is a Fredholm complex with Poincar\'e duality.
\end{teo}

\begin{proof}
The thesis of the theorem follows immediately from the previous theorem and from Theorem \ref{cheafa}.
\end{proof}

As a particular and important case we have the following two theorems:
 
\begin{teo}
\label{mario}
Let $(M,g)$ be an open, oriented and incomplete riemannian manifold of dimension $m$. Then the complexes $$(L^{2}\Omega^{*}(M,g),d_{max,*})\ and\  (L^{2}\Omega^{*}(M,g),d_{min,*})$$ are a pair of complementary Hilbert complexes. \\In particular if $\im(\overline{H}^{i}_{2,min}(M,g)\stackrel{i^{*}_{r,i}}{\longrightarrow}\overline{H}^{i}_{2,max}(M,g))$ is finite dimensional for each $i$ then $$\im(\overline{H}^{i}_{2,min}(M,g)\stackrel{i^{*}_{r,i}}{\longrightarrow}\overline{H}^{i}_{2,max}(M,g))$$ is a finite sequence of finite dimensional vector spaces with Poincar\'e duality.
\end{teo}

\begin{proof}
Let $*:\Lambda^i(M)\rightarrow \Lambda^{n-i}(M)$ the Hodge star operator. Then $*$ induces a map between $\Omega_{c}^i(M)$ and $\Omega_{c}^{n-i}(M)$ such that for $\eta, \omega\in \Omega_{c}^i(M)$ we have: $$<*\eta,* \omega>_{L^2\Omega^{n-i}(M,g)}=\int_{M}<*\eta,*\omega>_{M}dvol_{M}=\int_{M}*\eta\wedge**\omega=\int_{M}\omega\wedge*\eta=$$$$=<\omega,\eta>_{L^2\Omega^i(M,g)}=<\eta,\omega>_{L^2\Omega^i(M,g)}$$ that is $*$ is an isometry between $\Omega^{i}_{c}(M)$ and $\Omega^{n-i}_{c}(M)$. This implies that $*$ extends to an isometry between $L^{2}\Omega^{i}(M,g)$ and $L^{2}\Omega^{n-i}(M,g)$. Now it is an immediate consequence of Definition \ref{nino} and  Definition \ref{reso} that $$*\*d_{min,i}=\pm\delta_{min,n-i-1}*\ \text{ and that}\  *d_{max,i}=\pm\delta_{max,n-i-1}*$$ and the sign depends only on the parity of the degree $i$. 
 So we can apply Theorem \ref{duality} and the assertion follows.
\end{proof}

\begin{rem}
The previous theorem  shows that pair of complementary Hilbert complexes appear naturally in riemannian geometry. In fact the Hodge star operator provides naturally  a family of duality maps and so, in this case,   we do not need  to assume their existence.
\end{rem}

\begin{teo}
\label{polcas}
Let $(M,g)$ be an open, oriented and incomplete riemannian manifold of dimension $n$. Suppose that for each $i=0,...,n$ $ran(d_{min,i})$ is closed in $L^2\Omega^{i+1}(M,g)$. Then  there exists a Hilbert complex $(L^2\Omega^i(M,g)),d_{\mathfrak{m},i})$ such that for each $i=0,...n$ $$\mathcal{D}(d_{min,i})\subset \mathcal{D}(d_{\mathfrak{m},i})\subset \mathcal{D}(d_{max,i}),$$$d_{max,i}$ is an extension of $d_{\mathfrak{m},i}$ which is an extension of $d_{min,i}$ and   $$H^i_{2,\mathfrak{m}}(M,g)=\im(H^{i}_{2,min}(M,g)\stackrel{i^{*}_{i}}{\longrightarrow}H^{i}_{2,max}(M,g))$$ where $H^i_{2,\mathfrak{m}}(M,g)$ is the cohomology of the Hilbert complex $(L^2\Omega^i(M,g),d_{\mathfrak{m},i})$. Finally, if\\ $(L^2\Omega^i(M,g),d_{max,i})$ or equivalently $(L^2\Omega^i(M,g),d_{min,i})$ is Fredholm, then $(L^2\Omega^i(M,g),d_{\mathfrak{m},i})$ is a Fredholm complex with Poincar\'e duality.
\end{teo}

\begin{proof}
Also in this case it follows immediately from the previous Theorem and from Theorem \ref{cheafa}.
\end{proof}

We have the following corollary which is a \textbf{Hodge theorem} for the $L^2-$cohomology groups $\im(H^{i}_{2,min}(M,g)\stackrel{i^{*}_{i}}{\longrightarrow}H^{i}_{2,max}(M,g))$:

\begin{cor}
\label{giove}
In the same assumptions of Theorem \ref{polcas}; Let $\Delta_{i}:\Omega^i_{c}(M)\rightarrow \Omega_{c}^i(M)$ be the Laplacian acting on the space of smooth compactly supported forms. Then there exists a self-adjoint extension  $\Delta_{\mathfrak{m},i}:L^2\Omega^i(M,g)\rightarrow L^2\Omega^i(M,g)$ with closed range such that $$Ker(\Delta_{\mathfrak{m},i})\cong \im(H^{i}_{2,min}(M,g)\stackrel{i^{*}_{i}}{\longrightarrow}H^{i}_{2,max}(M,g)).$$ Moreover, if $(L^2\Omega^i(M,g),d_{max,i})$ or equivalently  $(L^2\Omega^i(M,g),d_{min,i})$ is Fredholm, then  $\Delta_{\mathfrak{m},i}$ is a Fredholm operator on its domain endowed with the graph norm.
\end{cor}

\begin{proof}
Consider the Hilbert complex $(L^2\Omega^i(M,g),d_{\mathfrak{m},i})$. For each $i=0,...,n$ define 
\begin{equation}
\label{rosmarino}
\Delta_{\mathfrak{m},i}:=d_{\mathfrak{m},i}^*\circ d_{\mathfrak{m},i}+d_{\mathfrak{m},i-1}\circ d_{\mathfrak{m},i-1}^*\end{equation}
with domain given by 
\begin{equation}
\label{sodio}
\mathcal{D}(\Delta_{\mathfrak{m},i})=\{\omega\in \mathcal{D}(d_{\mathfrak{m},i})\cap \mathcal{D}(d_{\mathfrak{m},i-1}^*):\ d_{\mathfrak{m},i}(\omega)\in \mathcal{D}(d_{\mathfrak{m},i}^*)\ \text{and}\ d_{\mathfrak{m},i-1}^*(\omega)\in \mathcal{D}(d_{\mathfrak{m},i-1})\}.
\end{equation}
In other words, for each $i=0,...,n$, $\Delta_{\mathfrak{m},i}$ is the $i-th$ Laplacian associated with the Hilbert complex $(L^2\Omega^i(M,g),d_{\mathfrak{m},i})$. So, as recalled in the first section,  it follows that  \eqref{rosmarino}  is a  self-adjoint operator. Moreover, by the fact that  $d_{min,i}$ has closed range for each $i=0,...,n$ it follows that also $\delta_{min,i}$ has closed range for each $i$. Finally this implies that also $d_{max,i}$ has closed range because $d_{max,i}=\delta_{min,i}^*$. This means   that for the Hilbert complex $(L^2\Omega^i(M,g),d_{\mathfrak{m},i})$ the $L^2-$cohomology and the reduced $L^2-$cohomology are exactly the same. The reason is that $\overline{ran(d_{\mathfrak{m},i})}=\overline{ran(d_{max,i})\cap Ker(d_{min,i+1})}$ $=ran(d_{max,i})\cap Ker(d_{min,i+1})$ because they are both closed in $L^2\Omega^{i+1}(M,g)$ and clearly $ran(d_{max,i})\cap Ker(d_{min,i+1})=ran(d_{\mathfrak{m},i})$. So we can apply \eqref{pppp} to get the first conclusion. Moreover by the fact that $ran(\Delta_{\mathfrak{m},i})=ran(d_{\mathfrak{m},i-1})\oplus ran(d^*_{\mathfrak{m},i})$ it follows that  $\Delta_{\mathfrak{m},i}$ is an operator with closed range. The reason of the previous equality is the following: clearly, by construction, we have always $ran(\Delta_{\mathfrak{m},i})\subset ran(d_{\mathfrak{m},i-1})\oplus ran(d_{\mathfrak{m},i}^*)$. 
Now let $\omega\in ran(d_{\mathfrak{m},i-1})\oplus ran(d_{\mathfrak{m},i}^*)$. Applying repeatedly the decomposition recalled in  Prop. \ref{beibei} and keeping in mind that $d_{\mathfrak{m},i}$ and  $d_{\mathfrak{m},i}^*$ have closed range for every $i$, we get that $$\omega=d_{\mathfrak{m},i-1}(d_{\mathfrak{m},i-1}^*(d_{\mathfrak{m},i-1}(\eta_1)))+d_{\mathfrak{m},i}^*(d_{\mathfrak{m},i}(d_{\mathfrak{m},i}^*(\eta_2)))$$ for some $\eta_1\in \mathcal{D}(d_{\mathfrak{m},i-1})$ and $\eta_2\in \mathcal{D}(d_{\mathfrak{m},i}^*).$  Clearly, by the construction of $\eta_1$ and $\eta_2$, we get that  $$d_{\mathfrak{m},i-1}(\eta_1)+d_{\mathfrak{m},i}^*(\eta_2)\in \mathcal{D}(\Delta_{\mathfrak{m},i})$$ and $$d_{\mathfrak{m},i-1}(d_{\mathfrak{m},i-1}^*(d_{\mathfrak{m},i-1}(\eta_1)))+d_{\mathfrak{m},i}^*(d_{\mathfrak{m},i}(d_{\mathfrak{m},i}^*(\eta_2)))=\Delta_{\mathfrak{m},i}(d_{\mathfrak{m},i-1}(\eta_1)+d_{\mathfrak{m},i}^*(\eta_2)).$$ Therefore we got $ran(\Delta_{\mathfrak{m},i})\supset ran(d_{\mathfrak{m},i-1})\oplus ran(d_{\mathfrak{m},i}^*)$ and in  this way we can conclude  that  $\Delta_{\mathfrak{m},i}$ is an operator with closed range.\\Finally, using  the fact that         $(L^2\Omega^i(M,g),d_{\mathfrak{m},i})$ is Fredholm, we get that  $\Delta_{\mathfrak{m},i}$ is self-adjoint, with finite dimensional nullspace and with closed range and therefore it is a Fredholm operator on its domain endowed with the graph norm.
\end{proof}

\begin{rem}
\label{europa}
From the previous proof  we get as a consequence that, under the assumptions of Theorem \ref{polcas}, the operator $d_{\mathfrak{m},i}$ has closed range for each $i$ and therefore for the Hilbert complex $(L^2\Omega^i(M,g),d_{\mathfrak{m},i})$ the $L^2-$cohomology coincides with the reduced $L^2-$cohomology.
\end{rem}

From now on we will focus our attention exclusively on the vector spaces\\ $\im(\overline{H}^{i}_{2,min}(M,g)\stackrel{i^{*}_{r,i}}{\longrightarrow}\overline{H}^{i}_{2,max}(M,g))$ because, using these, we will get some geometric and topological applications concerning the manifold $M$.\\ Anyway it will be  clear that all the following corollaries of the remaining part of this subsection apply also for the vector spaces $\im(\overline{H}^{i}_{2,min}(M,E_{*})\stackrel{i^{*}_{r,i}}{\longrightarrow}\overline{H}^{i}_{2,max}(M,E_{*}))$ under the hypothesis of theorem \ref{mzzm}. \vspace{1 cm}

\noindent
Now, to get a lighter notation,  we label the vector spaces $$\im(\overline{H}^{i}_{2,min}(M,g)\stackrel{i^{*}_{r,i}}{\longrightarrow}\overline{H}^{i}_{2,max}(M,g)):= \overline{H}^{i}_{2,m\rightarrow M}(M,g)\ \text{and}\ H^{i}_{2,m\rightarrow M}(M,g)$$ in the non-reduced case. Moreover, when it makes sense, we define 
\begin{equation}
\overline{\chi}_{2,m\rightarrow M}(M,g):=\sum_{i=0}^{m}(-1)^idim(\overline{H}^{i}_{2,m\rightarrow M}(M,g))
\end{equation}
and in the non-reduced case :
\begin{equation}
\chi_{2,m\rightarrow M}(M,g):=\sum_{i=0}^{m}(-1)^idim(H^{i}_{2,m\rightarrow M}(M,g))
\end{equation}
 
We have the following propositions:
\begin{prop}
\label{frengo}
In the hypothesis of  Theorem \ref{mario}, if $m$ is odd then:
\begin{equation}
\overline{\chi}_{2,m\rightarrow M}(M,g)=0.
\label{babbo} 
\end{equation}
Finally if    $(L^{2}\Omega^{i}(M,g),d_{max,i})$ is Fredholm, or equivalently if $(L^{2}\Omega^{i}(M,g),d_{min,i})$ is Fredholm,  the above results holds for $\chi_{2,m\rightarrow M}(M,g)$.
\end{prop}

\begin{proof}
The equality \eqref{babbo} is an immediate consequence of Theorem \ref{mario}. Finally, if for example  $(L^{2}\Omega^{i}(M,g),d_{max,i})$ is Fredholm then $H^{i}_{2,max}(M,g)\cong\overline{H}^{i}_{2,max}(M,g)\cong \overline{H}^{n-i}_{2,min}\cong H^{n-i}_{2,min}(M,g)$ and so also $(L^{2}\Omega^{i}(M,g),d_{min,i})$ is Fredholm. Obviously the same arguments show that,  if    $(L^{2}\Omega^{i}(M,g),d_{min,i})$ is Fredholm, then also  $(L^{2}\Omega^{i}(M,g),d_{max,i})$ is Fredholm and therefore in  \eqref{babbo} we can use $\chi_{2,m\rightarrow M}(M,g)$.
\end{proof}

\begin{prop} 
\label{gennaro}
In the hypothesis of Theorem \ref{duality}. Suppose that one of the two following properties is satisfied:
\begin{enumerate}
\item$ i_{r,i}^*:\overline{H}^i_{2,min}(M,g)\longrightarrow \overline{H}^{i}_{2,max}(M,g)$ is injective for all $i=0,...,n$,
\item $ i_{r,i}^*:\overline{H}^i_{2,min}(M,g)\longrightarrow \overline{H}^{i}_{2,max}(M,g)$  is surjective for all $i=0,...,n$.
\end{enumerate}
Then 
\begin{equation}
\overline{H}^i_{2,min}(M,g),\ \overline{H}^i_{2,max}(M,g)\ i=0,...,n
\end{equation}
 both are finite sequences of finite dimensional vector spaces with Poincar\'e duality. Finally, if one of the two properties above holds and if one of the two complexes  $(L^{2}\Omega^{i}(M,g),d_{max/min,i})$ is Fredholm, then the same conclusion holds for $$H^i_{2,min}(M,g),\ H^i_{2,max}(M,g)\ i=0,...,n.$$
\end{prop}

\begin{proof} Assume that $ i_{r,i}^*:\overline{H}^i_{2,min}(M,g)\longrightarrow \overline{H}^{i}_{2,max}(M,g)$ is injective for all $i=0,...,n$. Then $\overline{H}^i_{2,min}(M,g)\cong \overline{H}^{i}_{2,m\rightarrow M} (M,g)$. This implies that each $\overline{H}^{i}_{2,min}(M,g)$ is finite dimensional and  therefore, using  Theorem \ref{mario},  we get $\overline{H}^{i}_{2,min}(M,g)\cong \overline{H}^{n-i}_{2,min}(M,g).$ Finally by the fact that the Hodge star operator induces an isomorphism between  $\overline{H}^{i}_{2,min}(M,g)$ and $\overline{H}^{n-i}_{2,max}(M,g)$ we get that $\overline{H}^{i}_{2,max}(M,g)$ is a finite sequence of finite dimensional vector spaces with Poincar\'e duality.\\ Assume now that $ i_{r,i}^*:\overline{H}^i_{2,min}(M,g)\longrightarrow \overline{H}^{i}_{2,max}(M,g)$  is surjective for all $i=0,...,n$.  Then $\overline{H}^i_{2,max}(M,g)\cong \overline{H}^{i}_{2,m\rightarrow M} (M,g)$ and this implies that $\overline{H}^i_{2,max}(M,g)$ is a finite sequence of finite dimensional vector spaces with Poincar\'e duality. Finally, using again the isomorphism induced by the Hodge star operator between  $\overline{H}^{i}_{2,min}(M,g)$ and $\overline{H}^{n-i}_{2,max}(M,g)$ we get the same conclusions for $\overline{H}^i_{2,min}(M,g)$.
\end{proof}

Finally we conclude the section with the following proposition; before  stating it we give some definitions: let
 \begin{equation}
\label{defre}
d_{\mathfrak{m}}+d_{\mathfrak{m}}^*:\bigoplus_{i=0}^{n} L^2\Omega^i(M,g)\longrightarrow \bigoplus_{i=0}^{n} L^2\Omega^i(M,g)
\end{equation} 
be the operator  defined as $d_{\mathfrak{m}}+d_{\mathfrak{m}}^*:=\bigoplus_{i=0}^n(d_{\mathfrak{m},i}+d_{\mathfrak{m},{i-1}}^*)$ where $d_{\mathfrak{m},i}$ is defined in Theorem \ref{polcas} and the  domain of \eqref{defre} is  $$
\mathcal{D}(d_{\mathfrak{m}}+d_{\mathfrak{m}}^*)=\bigoplus_{i=0}^n\mathcal{D}(d_{\mathfrak{m},i}+d_{\mathfrak{m},i-1}^*)$$ and $\mathcal{D}(d_{\mathfrak{m},i}+d_{\mathfrak{m},i-1}^*)=\mathcal{D}(d_{\mathfrak{m},i})\cap \mathcal{D}(d_{\mathfrak{m},i-1}^*).$

\begin{prop}
\label{rtlrtl}
Let $(M,g)$ be an open oriented and incomplete riemannian manifold of dimension $n$. Suppose that for each $i=0,...,n$ $ran(d_{min,i})$ is closed in $L^2\Omega^{i+1}(M,g)$ and that  $(L^2\Omega^i(M,g),d_{\mathfrak{m},i})$ is a Fredholm complex. Then the operator $(d_{\mathfrak{m}}+d_{\mathfrak{m}}^*)_{ev}$  defined as $$d_{\mathfrak{m}}+d_{\mathfrak{m}}^*:\bigoplus_{i=0}^{n} L^2\Omega^{2i}(M,g)\longrightarrow \bigoplus_{i=0}^{n} L^2\Omega^{2i+1}(M,g)$$ with domain given by $$\mathcal{D}((d_{\mathfrak{m}}+d_{\mathfrak{m}}^*)_{ev}):=\bigoplus_{i=0}^n\mathcal{D}(d_{\mathfrak{m},2i}+d_{\mathfrak{m},2i-1}^*)$$
is a Fredholm operator on its domain endowed with the graph norm and its index satisfies 
\begin{equation}
\label{lalala}
\ind((d_{\mathfrak{m}}+d_{\mathfrak{m}}^*)_{ev})=\chi_{m\rightarrow M}(M,g)
\end{equation}
\end{prop}

\begin{proof}
By the fact that $(L^2\Omega^i(M,g),d_{\mathfrak{m},i})$ is a Fredholm complex it follows that $d_{\mathfrak{m}}+d_{\mathfrak{m}}^*$ is a Fredholm operator on its domain endowed with graph norm. Now if we define $(d_{\mathfrak{m}}+d_{\mathfrak{m}}^*)_{odd}$ analogously to $(d_{\mathfrak{m}}+d_{\mathfrak{m}}^*)_{ev}$, then it is clear that $\mathcal{D}(d_{\mathfrak{m}}+d_{\mathfrak{m}}^*)$ $=\mathcal{D}((d_{\mathfrak{m}}+d_{\mathfrak{m}}^*)_{ev})\oplus \mathcal{D}((d_{\mathfrak{m}}+d_{\mathfrak{m}}^*)_{odd})$, that $Ker(d_{\mathfrak{m}}+d_{\mathfrak{m}}^*)$ $=Ker((d_{\mathfrak{m}}+d_{\mathfrak{m}}^*)_{ev})\oplus Ker((d_{\mathfrak{m}}+d_{\mathfrak{m}}^*)_{odd})$  and that $ran(d_{\mathfrak{m}}+d_{\mathfrak{m}}^*)$ $=ran((d_{\mathfrak{m}}+d_{\mathfrak{m}}^*)_{ev})\oplus ran((d_{\mathfrak{m}}+d_{\mathfrak{m}}^*)_{odd})$. This implies immediately that also  $(d_{\mathfrak{m}}+d_{\mathfrak{m}}^*)_{ev}$ is a Fredholm operator on its domain endowed with the graph norm. Finally \eqref{lalala} is an easy consequence of the Hodge Theorem stated in Corollary \ref{giove}.
\end{proof}

\subsection{A topological obstruction to the existence of riemannian metric with finite $L^{2}-$cohomology}
Now we want to show another application of the vector spaces $ \overline{H}^{i}_{2,m\rightarrow M} (M,g)$. Consider again the complex  $(\Omega_{c}^{*}(M),d_{*})$. We will call a \textbf{ closed extension} of  $(\Omega_{c}^{*}(M),d_{*})$ any Hilbert complex $(L^{2}\Omega^{i}(M,g), D_{i})$ where $D_{i}:L^{2}\Omega^{i}(M,g)\rightarrow L^{2}\Omega^{i+1}(M,g)$ is a densely defined, closed operator which extends $d_{i}:\Omega^{i}_{c}(M,g)\rightarrow \Omega^{i+1}_{c}(M,g)$ and such that the action of $D_{i}$ on $\mathcal{D}(D_{i})$, its domain, coincides with the action of $d_{i}$ on $\mathcal{D}(D_i)$ in a distributional way. Obviously for every closed extension of  $(\Omega_{c}^{*}(M),d_{*})$ we have $(L^{2}\Omega^{*}(M,g),d_{min,*})\subseteq (L^{2}\Omega^{*}(M,g),D_{i})\subseteq (L^{2}\Omega^{*}(M,g),d_{max,*}). $ We will label with $\overline{H}^{i}_{2,D_{*}}(M,g)$,  $H^{i}_{2,D_{*}}(M,g)$ respectively the reduced cohomology and the cohomology groups of $(L^{2}\Omega^{i}(M,g), D_{i})$ and with $\mathcal{H}^{i}_{D_{*}}(M,g)$ its Hodge cohomology groups. Moreover if $(L^{2}\Omega^{*}(M,g),D'_{i})$ is another closed extension of  $(\Omega_{c}^{*}(M),d_{*})$ such that $(L^{2}\Omega^{*}(M,g),D_{i})\subseteq (L^{2}\Omega^{*}(M,g),D'_{i})$ we will label with $H^{i}_{2,D\rightarrow D'}(M,g), \overline{H}^{i}_{2,D\rightarrow D'}(M,g)$  respectively the image of the cohomology groups, reduced cohomology groups, of the complex  $(L^{2}\Omega^{*}(M,g),D_{i})$ into the cohomology groups, reduced cohomology groups, of the complex  $(L^{2}\Omega^{*}(M,g),D_{i}')$  induced by the natural inclusion of complexes.\\
Before we proceed we need  the following propositions.

\begin{prop}
Let (M,g) be an incomplete and oriented riemannian manifold of dimension $m$. For each $i=0,...,m$ consider $\mathcal{D}(d_{max,i})$. Let $\omega\in \mathcal{D}(d_{max,i}) $. Then there exists a sequence of smooth forms $\{\omega_{j}\}_{j\in \mathbb{N}}\subset \Omega^{i}(M)\cap L^{2}\Omega^{i}(M,g)$ such that :
\begin{enumerate}
\item $d_{i}\omega_{j}\in L^{2}\Omega^{i+1}(M,g)$.
\item $\omega_{j}\rightarrow \omega$ in $L^{2}\Omega^{i}(M,g)$.
\item $d_{i}\omega_{j}\rightarrow d_{max,i}\omega$ in $L^{2}\Omega^{i+1}(M,g)$.
\end{enumerate}
\label{gigino}
\end{prop}

\begin{proof}
See \cite{C} pag 93.
\end{proof}

 The next proposition is a variation of a result of de Rham, see \cite{dR} Theorem 24.

\begin{prop}
Let (M,g) be an incomplete and oriented riemannian manifold of dimension $m$. For each $i=0,...,m$ consider $\mathcal{D}(d_{max,i})$. Let $\omega\in \overline{ran(d_{max,i})}$ be  such that $\omega$ is smooth. Then there exists $\eta\in \Omega^{i}(M)$ such that $d_{i}\eta=\omega$.
\label{luca}
\end{prop}

\begin{proof}
By Poincar\'e duality between de Rham cohomology and compactly supported de Rham cohomology on an open and oriented manifold we know that it sufficient to show that $$\int_{M}\omega\wedge\phi=0$$ for each closed and compactly supported $n-i-1$form $\phi$ to get that $\omega$ is an exact $i+1-$form in the smooth de Rham complex.  Now, by Proposition \ref{gigino}, we know that there exists a sequence of smooth $i-$forms $\{\eta_{j}\}_{j\in \mathbb{N}}$ such that $d_{i}\eta_{j}\rightarrow \omega$ in $L^{2}\Omega^{i+1}(M,g).$ Then: $$\int_{M}\omega\wedge\phi=\int_{M}(\lim_{j\rightarrow \infty}d_{i}\eta_{j})\wedge\phi=\lim_{j\rightarrow \infty}\int_{M}d_{i}\eta_{j}\wedge\phi=0$$ by Stokes Theorem. 
\end{proof}

\begin{prop} Let $(L^{2}\Omega^{i}(M,g), D_{i})$ be any closed extension of $(\Omega_{c}^{*}(M),d_{*})$ where $(M,g)$ is an incomplete oriented  riemannian manifold. Then every cohomology class in $\overline{H}^{i}_{2,D_{*}}(M,g)$ has a smooth representative. The same conclusion holds for every cohomology class in $H^{i}_{2,D_{*}}(M,g)$.
\label{chicago}
\end{prop}

\begin{proof}
By  \eqref{pppp} we know that every cohomology class in $\overline{H}^{i}_{2,D_{*}}(M,g)$ has a representative in $\mathcal{H}^{i}_{D_{*}}(M,g)$. Now, by elliptic regularity (see for example de Rham book \cite{dR}), it follows that every element in $\mathcal{H}^{i}_{D_{*}}(M,g)$ is smooth. Now if we look at Proposition \ref{dfff}, elliptic regularity tells us again that every element in $\mathcal{D}^{\infty}(L^{2}\Omega^i(M,g))$ is  smooth. Therefore from this we get immediately the statement  for   $H^{i}_{2,D_{*}}(M,g)$.
\end{proof}

From the above Propositions \ref{luca} and \ref{chicago} it follows that that there exists a well defined map from $\overline{H}^i_{2,D_*}(M,g)$, respectively from $H^i_{2,D_*}(M,g)$, to the ordinary de Rham cohomology of $M$ which assigns to each cohomology class  $[\omega]\in \overline{H}^{i}_{2,D_*}(M,g)$, respectively $[\omega]\in H^{i}_{2,D_*}(M,g)$, the cohomology class  in $H^{i}_{dR}(M)$ given  by  the smooth representatives of $[\omega].$ By Proposition \ref{luca} this cohomology class in $H^{i}_{dR}(M)$ does not depend on the choice of the smooth representative of $[\omega]$ and therefore this map is well defined. \\
We will label these maps:
\begin{equation}
s^{*}_{D,i}: H^{i}_{2,D_*}(M,g)\longrightarrow H^{i}_{dR}(M)\ \text{in the non-reduced case}
\label{pollok}
\end{equation}
and 
\begin{equation}
s^{*}_{r,D,i}: \overline{H}^{i}_{2,D_*}(M,g)\longrightarrow H^{i}_{dR}(M)\ \text{in the reduced case}
\label{apollok}
\end{equation}

In particular for the maximal and minimal extension we will label these maps:
\begin{equation}
s^{*}_{M,i}: H^{i}_{2,max}(M,g)\longrightarrow H^{i}_{dR}(M)\ \text{in the non-reduced case}
\label{pollo}
\end{equation}
and 
\begin{equation}
s^{*}_{r,M,i}: \overline{H}^{i}_{2,max}(M,g)\longrightarrow H^{i}_{dR}(M)\ \text{in the reduced case}
\label{apollo}
\end{equation}
and analogously for the minimal extension
\begin{equation}
s^{*}_{m,i}: H^{i}_{2,min}(M,g)\longrightarrow H^{i}_{dR}(M)\ \text{in the non-reduced case}
\label{pollom}
\end{equation}
and 
\begin{equation}
s^{*}_{r,m,i}: \overline{H}^{i}_{2,min}(M,g)\longrightarrow H^{i}_{dR}(M)\ \text{in the reduced case}
\label{apollom}
\end{equation}

Now we are ready to state  the following  proposition:
\begin{prop}
\label{toposodo}
Let $(M,g)$ be an open, oriented and incomplete riemannian manifold. Let $(L^{2}\Omega^{*}(M,g), D_{a,*}),(L^{2}\Omega^{*}(M,g), D_{b,*})$ be two closed extensions of $(\Omega_{c}^{*}(M),d_{*})$ such that  
\begin{equation}
(L^{2}\Omega^{*}(M,g),d_{min,*})\subseteq (L^{2}\Omega^{*}(M,g),D_{a,*})\subseteq  (L^{2}\Omega^{*}(M,g),D_{b,*})\subseteq  (L^{2}\Omega^{*}(M,g),d_{max,*}).
\label{paolone}
\end{equation}
Then the two following diagrams commute:
\begin{equation}
\label{aappaa}
\xymatrix{\ar @{} [dr] |{}
H_{c}^i(M)\ar[d] \ar[r] & H^i_{dR}(M)  \\
H^i_{2,min}(M,g)\ar[d]\ar[r] & H^i_{2,max}(M,g) \ar[u]_{s^*_{M,i}}\\
H^i_{2,D_{a,*}}(M,g)\ar[r]  & H^i_{2,D_{b,*}}(M,g) \ar[u]}
\xymatrix{\ar @{} [dr] |{}
H_{c}^i(M)\ar[d] \ar[r] & H^i_{dR}(M)  \\
\overline{H}^i_{2,min}(M,g)\ar[d]\ar[r] & \overline{H}^i_{2,max}(M,g) \ar[u]_{s^*_{r,M,i}}\\
\overline{H}^i_{2,D_{a,*}}(M,g)\ar[r]  & \overline{H}^i_{2,D_{b,*}}(M,g) \ar[u]}
\end{equation}
where all the above arrows without label are the natural maps between cohomology, respectively reduced cohomology groups, induced by the natural inclusion of the relative complexes.
\end{prop}

\begin{proof}
It is clear that both the two following diagrams commute:
$$ 
\xymatrix{\ar @{} [dr] |{}
H_{c}^i(M)\ar[d] \ar[rd]   \\
H^i_{2,min}(M,g)\ar[d]\ar[r] & H^i_{2,max}(M,g)\\
H^i_{2,D_{a,*}}(M,g)\ar[r]  & H^i_{2,D_{b,*}}(M,g) \ar[u]}
\xymatrix{\ar @{} [dr] |{}
H_{c}^i(M)\ar[d] \ar[rd]  \\
\overline{H}^i_{2,min}(M,g)\ar[d]\ar[r] & \overline{H}^i_{2,max}(M,g) \\
\overline{H}^i_{2,D_{a,*}}(M,g)\ar[r]  & \overline{H}^i_{2,D_{b,*}}(M,g) \ar[u]}
$$
So, to complete the proof, we have to show that the two following diagrams are both commutative:
$$
\xymatrix{\ar @{} [dr] |{}
H^i_{c}(M) \ar[d] \ar[d]\ar[rd]   \\
H^i_{2,max}(M,g) \ar[r]^{s^*_{M,i}} & H^i_{dR}(M)  }
\xymatrix{\ar @{} [dr] |{}
H^i_{c}(M) \ar[d] \ar[d]\ar[rd]   \\
\overline{H}^i_{2,max}(M,g) \ar[r]^{s^*_{r,M,i}} & H^i_{dR}(M)  }
$$
To prove this it is enough to show that given an $i-$form $\omega$ which is closed, smooth and with compact support, if $[\omega]=0$ in  $H^i_{2,max}(M,g)$ or in $\overline{H}^i_{2,max}(M,g)$ then also $s^*_{M,i}(\omega)=0$, respectively $s^*_{r,M,i}(\omega)=0$, that is the cohomology class of $\omega$ in $H^{i}_{dR}(M)$ is null. This last statement follows immediately from Proposition \ref{luca}.
\end{proof}

Using the previous proposition we get the following corollary in which the first statement  extends a result of Anderson, see \cite{MA}, to the  case of an incomplete riemannian metric both for the reduced and the unreduced $L^2-$cohomology groups.

\begin{cor}
\label{loplp}
 Let $(M,g)$ be as in the previous proposition. Then from \eqref{aappaa} we get these two  commutative diagrams whose arrows   are injective maps:

\begin{equation}
\label{acenadalgreco}
\xymatrix{\ar @{} [dr] |{}
\im(H^j_{c}(M)\rightarrow H^j_{dR}(M)) \ar[d] \ar[d]\ar[rd]   \\
\overline{H}^j_{2,m\rightarrow M}(M,g) \ar[r] &  \overline{H}^j_{2,D_{a}\rightarrow D_{b}}(M,g) }
\xymatrix{\ar @{} [dr] |{}
\im(H^j_{c}(M)\rightarrow H^j_{dR}(M)) \ar[d] \ar[d]\ar[rd]   \\
H^j_{2,m\rightarrow M}(M,g) \ar[r] & H^j_{2,D_{a}\rightarrow D_{b}}(M,g) }
\end{equation}
Moreover if $H^i_{c}(M)\rightarrow H^i_{dR}(M)$ is injective then: 
\begin{equation}
\label{bnmn}
H^i_{c}(M)\rightarrow H^i_{2,m\rightarrow M}(M,g),\ H^i_{c}(M)\rightarrow \overline{H}^i_{2,m\rightarrow M}(M,g)
\end{equation}
are injective and therefore for each closed extension $(L^{2}\Omega^*(M,g),D_{*})$ also the following maps are injective:
\begin{equation}
\label{bnmnb}
H^i_{c}(M)\rightarrow H^i_{2,D}(M,g),\ H^i_{c}(M)\rightarrow \overline{H}^i_{2,D}(M,g)
\end{equation}
\end{cor}

\begin{proof}
It is an immediate consequence of the previous proposition.
\end{proof}

Now we give  other three corollaries of  Proposition \ref{toposodo}. In particular the third corollary shows that there could exist a \textbf{topological obstruction} to the existence of a riemannian metric on $g$  with certain analytic properties.

\begin{cor}
Let $M$  be an open manifold such that for some $j$ $im(H_{c}^{j}(M)\stackrel{i^*_{j}}{\rightarrow}H_{dR}^j(M))$ is non-trivial. Then for every riemannian metric $g$  on $M$ and for every pair of closed extensions $(L^{2}\Omega^{*}(M,g), D_{a,*})$, $(L^{2}\Omega^{*}(M,g), D_{b,*})$ such that       $ (L^{2}\Omega^{*}(M,g), D_{a,*})  \subseteq (L^{2}\Omega^{*}(M,g), D_{b,*})$ we have that for the same $j$ both vector spaces $$H^{j}_{2,D_{a}\rightarrow D_{b}}(M,g),\  \overline{H}^{j}_{2,D_{a}\rightarrow D_{b}}(M,g)$$ are non-trivial.
In particular this implies that for the same $j$ the following four vector spaces are non-trivial: $$ H^{j}_{2,D_{a}}(M,g),\  H^{j}_{2,D_{b}}(M,g),\ \overline{H}^{j}_{2,D_{a}}(M,g),\  \overline{H}^{j}_{2,D_{b}}(M,g) .$$
\end{cor}

\begin{cor}
Let $(M,g)$ be an open, oriented and incomplete riemannian manifold. Suppose that there exists a pair of closed extensions $(L^{2}\Omega^{*}(M,g), D_{a,*}), (L^{2}\Omega^{*}(M,g), D_{b,*}) $ of $(\Omega_{c}^*(M),d_{*})$ such  that they are both weak Fredholm and $ (L^{2}\Omega^{*}(M,g), D_{a,*})\subseteq (L^{2}\Omega^{*}(M,g), D_{b,*}) $. \\Then  $ im(H^{j}_{c}(M)\stackrel{i^*_{j}}{\longrightarrow}H^{j}_{dR}(M))$   is finite dimensional and we have 
\begin{equation}
dim( \im(H^{j}_{c}(M)\stackrel{i^*_{j}}{\longrightarrow}H^{j}_{dR}(M)))\leq dim\overline{H}^{j}_{2,D_{a}}(M,g)
\end{equation}

\begin{equation}
dim( \im(H^{j}_{c}(M)\stackrel{i^*_{j}}{\longrightarrow}H^{j}_{dR}(M)))\leq dim\overline{H}^{j}_{2,D_{b}}(M,g)
\end{equation}

In particular if one of the two complexes $ (L^{2}\Omega^{*}(M,g),d_{max/min,*})$  is weak Fredholm then also the other one is weak Fredholm and for each $j=0,...,m$ we have: 
\begin{equation}
dim( \im(H^{j}_{c}(M)\stackrel{i^*_{j}}{\longrightarrow}H^{j}_{dR}(M)))\leq dim\overline{H}^{j}_{2,max}(M,g)
\end{equation}
\begin{equation}
 dim( \im(H^{j}_{c}(M)\stackrel{i^*_{j}}{\longrightarrow}H^{j}_{dR}(M)))\leq dim\overline{H}^{j}_{2,min}(M,g).
\end{equation}
Finally if one of the two complexes $ (L^{2}\Omega^{*}(M,g),d_{max/min,*})$  is  Fredholm then for each $j=0,...,m$ we have: 
\begin{equation}
dim( \im(H^{j}_{c}(M)\stackrel{i^*_{j}}{\longrightarrow}H^{j}_{dR}(M)))\leq dimH^{j}_{2,max}(M,g)
\end{equation}
\begin{equation}
 dim( \im(H^{j}_{c}(M)\stackrel{i^*_{j}}{\longrightarrow}H^{j}_{dR}(M)))\leq dimH^{j}_{2,min}(M,g).
\end{equation}
\label{pinotto}
\end{cor}

\begin{proof} It is an immediate consequence of Corollary \ref{loplp}.
\end{proof}

Now we conclude this subsection with the following corollary. We will use it later in Prop. \ref{lastlast} and in Cor. \ref{lastlastlast} to show some examples of open and oriented manifolds which does not admit a riemannian metric with finite $L^2-$cohomology or  with finite reduced $L^2-$cohomology.
\begin{cor}
\label{gianni}
Let $M$ be an open, oriented and incomplete riemannian manifold where $m=dim(M)$. Suppose that for some $j\in \{0,...,m\}$  $ \im(H^{j}_{c}(M)\stackrel{i^*_j}{\longrightarrow}H^{j}_{dR}(M))$ is infinite dimensional. Then on $M$ there is no riemannian metric $g$ (complete or incomplete) such that, for some closed extension $(L^{2}\Omega^{*}(M,g), D_{*})$ of $(\Omega^{*}_{c}(M),d_{*})$, one of the following properties is satisfied:

\begin{enumerate}
\item  $\overline{H}^{j}_{2,D_{*}}(M,g)$ or $\overline{H}^{m-j}_{2,D_{*}}(M,g)$ is finite dimensional.
\item  $H^{j}_{2,D_{*}}(M,g)$ or $H^{m-j}_{2,D_{*}}(M,g)$ is finite dimensional.
\item $D_{j}^*\circ D_{j}+D_{j-1}\circ D_{j-1}^*$ on its domain (as defined in \eqref{saed}) endowed with the graph norm is a Fredholm operator.
\end{enumerate}
Moreover on $M$ there is no  riemannian metric $g$ such that:
\begin{enumerate}
\item $\Delta_{max,j}$, the maximal closed extension of $\Delta_{j}:\Omega_{c}^j(M)\rightarrow \Omega^j_{c}(M)$,  has finite dimensional nullspace.
\item  $\Delta_{min,j}$, the minimal closed extension of $\Delta_{j}:\Omega_{c}^j(M)\rightarrow \Omega^j_{c}(M)$,  satisfies\\ $dim(ran(\Delta_{min,j})^{\bot})<\infty$.
\end{enumerate}

\end{cor}

\begin{proof}
The first two points are immediate consequence of Corollary \ref{loplp} and Theorem \ref{mario}. The third  point follows immediately by \eqref{said} and \eqref{pppp}.
Finally, for the last two points , if $Ker(\Delta_{max,j})$ is finite dimensional then all the other closed extensions of $\Delta_{j}:\Omega^j_{c}(M)\rightarrow \Omega^j_{c}(M)$ have finite dimensional nullspace. So we can apply the third point to get the conclusion. Finally if we consider $\Delta_{min,j}$ then we have $\Delta_{min,j}^*=\Delta_{max,j}$. So if  $dim(ran(\Delta_{min,j})^{\bot})<\infty$  then $Ker(\Delta_{max,j})$ is finite dimensional. Now by the previous point we can get the conclusion. 
\end{proof}

\subsection{$L^2$ and topological signature on an open oriented and incomplete riemannian manifold.}

The aim of this subsection is to show that if $(M,g)$ is an open oriented and incomplete riemannian manifold such that  for  $i=2k$  $\overline{H}^{i}_{2,m\rightarrow M}(M,g)$ is  finite dimensional, where $4k=dimM$, then we can define over $M$ an $L^{2}-$signature and a topological signature.  The first step is to show  that using the wedge product we can construct a well defined and non-degenerate pairing between $\overline{H}^{i}_{2,m\rightarrow M}(M,g)$ and $\overline{H}^{n-i}_{2,m\rightarrow M}(M,g)$ where $n=dimM$. In fact any cohomology class $[\omega]\in \overline{H}^{i}_{2,m\rightarrow M}(M,g) $ is a cohomology class in $\overline{H}^{i}_{2,max}(M,g)$ which admits a representative in $Ker(d_{min,i})$. So we can define:
\begin{equation}
\overline{H}^{i}_{2,m\rightarrow M}(M,g)\times \overline{H}^{n-i}_{2,m\rightarrow M}(M,g)\longrightarrow \mathbb{R},\  ([\eta],[\omega])\mapsto \int_{M} \eta\wedge \omega
\label{cambo}
\end{equation}
where $\omega\in Ker(d_{min,i})$ and $\eta\in Ker(d_{min,n-i})$

\begin{prop}
Let (M,g) be an open, oriented and incomplete riemannian manifold of dimension $n$.
Then \eqref{cambo} is a well defined and  non degenerate pairing and therefore when the vector spaces $\overline{H}^{i}_{2,m\rightarrow M}(M,g)\ i=0,...,n$ are finite dimensional it induces an isomorphism between $$\overline{H}^{i}_{2,m\rightarrow M}(M,g)\ and \ (\overline{H}^{n-i}_{2,m\rightarrow M}(M,g))^*.$$
\end{prop}

\begin{proof}
The first step is to show that \eqref{cambo} is well defined. Let $\eta',\ \omega'$  be other two forms such that $[\eta]=[\eta']$ in $\overline{H}^{i}_{2,m\rightarrow M}(M,g)$,   $[\omega]=[\omega']$ in $\overline{H}^{n-i}_{2,m\rightarrow M}(M,g)$ and that $\omega'\in Ker(d_{min,i})$, $\eta'\in Ker(d_{min,n-i})$ . Then there exist $\alpha \in \overline{d_{max,i-1}}\cap \mathcal{D}(d_{min,i})$ and  $\beta \in \overline{d_{max,n-i-1}}\cap \mathcal{D}(d_{min,n-i})$ such that $\eta =\eta'+\alpha$ and $\omega=\omega'+\beta$. Therefore: $$\int_{M}\eta\wedge\omega=\int_{M}(\eta'+\alpha)\wedge(\omega'+\beta)=\int_{M}\eta'\wedge\omega'+\int_{M}\eta'\wedge\beta+\int_{M}\alpha\wedge\omega'+\int_{M}\alpha\wedge\beta$$
Now $$\int_{M}\eta'\wedge\beta=\pm\int_{M}<\eta',*\beta>dvol_{M}=<\eta',*\beta>_{L^{2}\Omega^i(M,g)}=0$$ because $Ker(d_{min,i})^{\bot}=\overline{ran(\delta_{max,i})}.$ In the same way: $$\int_{M}\alpha\wedge\beta=\pm\int_{M}<\alpha,*\beta>dvol_{M}=<\alpha,*\beta>_{L^{2}\Omega^i(M,g)}=0.$$ Finally 
$$\int_{M}\alpha\wedge\omega'=\pm\int_{M}<\alpha,*\omega'>dvol_{M}=<\alpha,*\omega'>_{L^{2}\Omega^i(M,g)}=0$$ because 
$Ker(\delta_{min,i-1})^{\bot}=\overline{ran(d_{max,i-1})}$. So we can conclude that \eqref{cambo} is well defined. Now fix $[\eta] \in \overline{H}^{i}_{2,m\rightarrow M}(M,g)$ and suppose that for each  $[\omega] \in \overline{H}^{n-i}_{2,m\rightarrow M}(M,g)$ the pairing \eqref{cambo} vanishes. Then this means that for each $\omega\in Ker(d_{min,n-i})$ we have $\int_{M}\eta\wedge \omega=0$. We also know that $\int_{M}\eta\wedge \omega=<\eta,*\omega>_{L^2\Omega^i(M,g)}$ and that  $*(Ker(d_{min,n-i}))=Ker(\delta_{min,i-1})$ . So by the fact that $(Ker(\delta_{min,i-1}))^{\bot}=\overline{ran(d_{max,i-1})}$ we obtain that $[\eta]=0$. In the same way if  $[\omega]\in \overline{H}^{n-i}_{2,m\rightarrow M}(M,g)$ is such that  for each  $[\eta] \in \overline{H}^{i}_{2,m\rightarrow M}(M,g)$ the pairing \eqref{cambo} vanishes then we know that  for each $\eta\in Ker(d_{max,i})$ we have $\int_{M}\eta\wedge \omega=0$. But we know that  $\int_{M}\eta\wedge \omega=<\eta,*\omega>_{L^{2}\Omega^i(M,g)}$. So by the fact that   $*(\overline{ran(d_{max,n-i-1}}))=\overline{ran(\delta_{max,i})}$  and  that $(Ker(d_{min,i}))^{\bot}=\overline{ran(\delta_{max,i})}$ we obtain that $[\omega]=0$.\\  So we can conclude that the pairing \eqref{cambo} is well defined and non-degenerate and therefore when the vector spaces $\overline{H}^{i}_{2,m\rightarrow M}(M,g)\ i=0,...,n$ are finite dimensional it induces an isomorphism between $\overline{H}^{i}_{2,m\rightarrow M}(M,g)$ and $(\overline{H}^{n-i}_{2,m\rightarrow M}(M,g))^*.$
\end{proof}

\begin{rem}
We can look at this proposition as an alternative statement (and proof) of Theorem \ref{mario}.
\end{rem}

We have the following immediate corollary:

\begin{cor}
Let $(M,g)$ be an open, oriented and incomplete riemannian manifold of dimension $4n$. Then on $\overline{H}^{2n}_{2,m\rightarrow M}(M,g)$ the pairing \eqref{cambo} is a symmetric bilinear form.
\end{cor}

 We can now state  the following definition:
\begin{defi}
\label{qwerq}
Let $(M,g)$ be an open and oriented riemannian manifold of dimension $4n$ such that, for  $i=2n$, $\overline{H}^{2n}_{2,m\rightarrow M}(M,g)$ is finite dimensional. Then we define the $L^{2}-$signature of $(M,g)$ and we label it $\sigma_{2}(M,g)$ as the signature of the pairing \eqref{cambo} on $\overline{H}^{2n}_{2,m\rightarrow M}(M,g)$.
\label{nnbb}
\end{defi} 

Consider now the sequence of vector spaces $\im(H^{i}_{c}(M)\rightarrow H_{dR}^{i}(M))\ i=0,...,dimM$. A cohomology class in $\im(H^{i}_{c}(M)\rightarrow H_{dR}^{i}(M))$ is a cohomology class in $H^i_{dR}(M)$ which admits as representative a smooth and closed form with compact support. So in a similar way to the previous case we can define:
\begin{equation}
\label{cambora}
\im(H^{i}_{c}(M)\rightarrow H^i_{dR}(M))\times \im(H^{n-i}_{c}(M)\rightarrow H^{n-i}_{dR}(M)) \longrightarrow \mathbb{R},\  ([\eta],[\omega])\mapsto \int_{M} \eta\wedge \omega
\end{equation}
where $\omega$ is an $i-$form closed with compact support and in the same way $\eta$ is a closed $n-i-$form with compact support. Now by Poincar\'e duality for open and oriented manifolds we get easily that this pairing is well defined and non-degenerate. So we can conclude that, if for each $i=0,...,dimM$ $\im(H^i_{c}(M)\rightarrow H_{dR}^i(M))$ is finite dimensional, then \eqref{cambora} induces an isomorphism between $\im(H^{i}_{c}(M)\rightarrow H^i_{dR}(M))$ and $\im(H^{n-i}_{c}(M)\rightarrow H^{n-i}_{dR}(M))^*$. Moreover it is clear that when $dimM=4n$ then, for $i=2n$,  \eqref{cambora} is a symmetric bilinear form. This implies that when $dimM=4n$  it is possible  to define a signature on $M$, which is topological by de Rham isomorphism theorem,  taking the signature of the pairing \eqref{cambora} for $i=2n$. This leads us to state the next  proposition:
\begin{prop}
\label{topomicio}
Let $(M,g)$ be an open, oriented and incomplete riemannian manifold of dimension $4n$. If $(M,g)$ admits the $L^2-$signature $\sigma_{2}(M,g)$ of Definition \ref{qwerq} then it admits also a topological signature defined  as the signature of the pairing \eqref{cambora} on $\im(H^{2n}_{c}(M)\rightarrow H^{2n}_{dR}(M))$.
\end{prop} 
\begin{proof} 
If $M$ admits the signature $\sigma_2(M,g)$ then, by Definition \ref{qwerq}, we know that $\overline{H}^{2n}_{2,m\rightarrow M}(M,g)$ is finite dimensional. Now, by Corollary \ref{loplp}, we know that also $\im(H_c^{2n}(M)\rightarrow H^{2n}(M))$ is finite dimensional and so  \eqref{cambora} admits a signature. 
\end{proof}
Moreover in the next section we will see that, on a class of open, incomplete and oriented riemannian manifold,  the $L^{2}-$signature of Definition \ref{qwerq} has a topological meaning.

\section{Topological Applications}

The aim of this section is to exhibit some topological and geometrical applications of the previous results. In the first part we show some applications to the intersection cohomology with general perversity of a compact and smoothly stratified pseudomanifold. In the last part we exhibit some examples for which Corollary \ref{gianni} applies.
\\To get the paper as self-contained as possible we will recall briefly the definitions of smoothly stratified pseudomanifold with a Thom-Mather stratification, quasi-edge metric and intersection cohomology with general perversity.

\subsection{A brief reminder on (smoothly) stratified pseudomanifolds and intersection cohomology}

We start this subsection by recalling the notions of a smoothly stratified pseudomanifold with a Thom-Mather stratification.  For the more general (and simple) definition of stratified pseudomanifold we refer to \cite{BA} and \cite{KW}.
\begin{defi}   
 A smoothly stratified pseudomanifold $X$ with a Thom-Mather stratification is a metrizable, locally compact, second countable space which admits a locally finite decomposition into a union of locally closed strata $\mathfrak{G}=\{Y_{\alpha}\}$, where each $Y_{\alpha}$ is a smooth, open and connected manifold, with dimension depending on the index $\alpha$. We assume the following:
\begin{enumerate}
\item If $Y_{\alpha}$, $Y_{\beta} \in \mathfrak{G}$ and $Y_{\alpha} \cap \overline{Y}_{\beta} \neq \emptyset$ then $Y_{\alpha} \subset \overline{Y}_{\beta}$
\item  Each stratum $Y$ is endowed with a set of control data $T_{Y} , \pi_{Y}$ and $\rho_{Y}$ ; here $T_{Y}$ is a neighborhood of $Y$ in $X$ which retracts onto $Y$, $\pi_{Y} : T_{Y} \rightarrow Y$
is a fixed continuous retraction and $\rho_{Y}: T_{Y}\rightarrow [0, 2)$ is a proper radial function in this tubular neighborhood such that $\rho_{Y}^{-1}(0) = Y$ . Furthermore,
we require that if $Z \in \mathfrak{G}$ and $Z \cap T_{Y}\neq \emptyset$  then
$(\pi_{Y} , \rho_{Y} ) : T_{Y} \cap Z \rightarrow Y\times [0,2) $
is a proper differentiable submersion.
\item If $W, Y,Z \in \mathfrak{G}$, and if $p \in T_{Y} \cap T_{Z} \cap W$ and $\pi_{Z}(p) \in T_{Y} \cap Z$ then
$\pi_{Y} (\pi_{Z}(p)) = \pi_{Y} (p)$ and $\rho_{Y} (\pi_{Z}(p)) = \rho_{Y} (p)$.
\item If $Y,Z \in \mathfrak{G}$, then
$Y \cap \overline{Z} \neq \emptyset \Leftrightarrow T_{Y} \cap Z \neq \emptyset$ ,
$T_{Y} \cap T_{Z} \neq \emptyset \Leftrightarrow Y\subset \overline{Z}, Y = Z\ or\ Z\subset \overline{Y} .$
\item  For each $Y \in \mathfrak {G}$, the restriction $\pi_{Y} : T_{Y}\rightarrow Y$ is a locally trivial fibration with fibre the cone $C(L_{Y})$ over some other stratified space $L_{Y}$ (called the link over $Y$ ), with atlas $\mathcal{U}_{Y} = \{(\phi,\mathcal{U})\}$ where each $\phi$ is a trivialization
$\pi^{-1}_{Y} (U) \rightarrow U \times C(L_{Y} )$, and the transition functions are stratified isomorphisms  which preserve the rays of each conic
fibre as well as the radial variable $\rho_{Y}$ itself, hence are suspensions of isomorphisms of each link $L_{Y}$ which vary smoothly with the variable $y\in U$.
\item For each $j$ let $X_{j}$ be the union of all strata of dimension less or equal than $j$, then $$X-X_{n-1}\ is\ dense\ in\ X$$
\end{enumerate}
\label{thom}
\end{defi}

The \textbf{depth} of a stratum $Y$ is largest integer $k$ such that there is a chain of strata $Y=Y_{k},...,Y_{0}$ such that $Y_{j}\subset \overline{Y_{j-1}}$ for $i\leq j\leq k.$ A stratum of maximal depth is always a closed subset of $X$.  The  maximal depth of any stratum in $X$ is called the \textbf{depth of $X$} as stratified spaces.
 Consider the filtration
\begin{equation}
X = X_{n}\supset X_{n-1}\supset X_{n-2}\supset X_{n-3}\supset...\supset X_{0}
\label{pippo}
\end{equation} 
 We refer to the open subset $X-X_{n-1}$ of a stratified pseudomanifold $X$ as its regular set, and the union of all other strata as the singular set,
$$reg(X):=X-sing(X)\ \text{where}\ sing(X):=\bigcup_{Y\in \mathfrak{G}, depthY>0 }Y. $$
For more details and properties we refer to \cite{ALMP}. \\Now we take from \cite{FB} the following definition and result.
Before giving the definition we recall that two riemannian metrics $g,h$ on a smooth manifold $M$ are  \textbf{quasi-isometric} if there are constants $c_{1}, c_{2}$ such that $c_{1}h\leq g\leq c_{2}h$.

\begin{defi}
\label{zedge}
 Let $X$ be a smoothly stratified pseudomanifold with a Thom-Mather stratification and let $g$ be a riemannian metric on $reg(X)$.     We call  $g$ a \textbf{ quasi-edge metric with weights} if it satisfies the following properties:
\begin{enumerate}
\item Take any stratum $Y$ of $X$; by   definition \ref{thom} for each $q\in Y$ there exists an open neighborhood $U$ of $q$ in $Y$ such that $\phi:\pi_{Y}^{-1}(U)\rightarrow U\times C(L_{Y})$ is a stratified isomorphism; in particular $\phi:\pi_{Y}^{-1}(U)\cap reg(X)\rightarrow U\times reg(C(L_{Y}))$ is a diffeomorphism. Then, for each $q\in Y$, there exists one of these trivializations $(\phi,U)$ such that $g$ restricted on $\pi_{Y}^{-1}(U)\cap reg(X)$ satisfies the following properties:
\begin{equation} 
(\phi^{-1})^{*}(g|_{\pi_{Y}^{-1}(U)\cap reg(X)})\cong dr\otimes dr+h_{U}+r^{2c}g_{L_{Y}}
\label{yhnn}
\end{equation}
 where $h_{U}$ is a riemannian metric defined over $U$, $c\in \mathbb{R}$ and $c>0$, $g_{L_{Y}}$ is a riemannian metric   on $reg(L_{Y})$, $dr\otimes dr+h_{U}+r^{2c}g_{L_{Y}}$ is a riemannian metric of product type on $U\times reg(C(L_{Y}))$ and with $\cong$ we mean \textbf{quasi-isometric}. 
 \item  If $p$ and $q$ lie in the same stratum $Y$ then in \eqref{yhnn} there is the same weight. We label it $c_{Y}$. 
\end{enumerate}
\end{defi}

\begin{rem} Implicit in the above definition is the fact that if the codimension of $Y$ is $1$ then $L_{Y}$ is just a point and therefore  $(\phi^{-1})^{*}(g|_{\pi_{Y}^{-1}(U)\cap reg(X)})\cong dr\otimes dr+h_{U}$.
\end{rem}

We refer to \cite{FB} for more comments about the above definitions, for some properties about metrics of this kind and for the proof of the following proposition. 

\begin{prop} Let $X$ be a smoothly stratified pseudomanifold with a Thom-Mather stratification $\mathfrak{X}$. For any stratum $Y\subset X$ fix a positive real number $c_{Y}$. Then there exists a quasi-edge metric with weights $g$ on $reg(X)$ having the numbers $\{c_{Y}\}_{Y\in \mathfrak{X}}$ as weights.
\label{top}
\end{prop}

Now we need to recall briefly the notion of intersection homology with general perversities. Intersection homology is a deep and rich field of algebraic topology founded by Mark Goresky and Robert MacPherson at the end of seventies. From the first two fundamental papers, \cite{GM} and \cite{GMA},  there have been several developments  and the original theory has been  extended in many directions. Our intention now is to recall briefly the extension of intersection homology given by Greg Friedman in \cite{IGP}.  For the original theory introduced by Goresky and MacPherson and also for more topological property about stratified pseudomanifolds we refer to \cite{GM}, \cite{GMA},  \cite{BA} and \cite{KW}.

\begin{defi} Let  $X$ be a compact and oriented stratified pseudomanifold of dimension $n$. A general perversity on $X$  is any function
\begin{equation}
 p:\{Singular\ Strata\ of\ X\}\rightarrow \mathbb{Z}.
\end{equation}
\label{ujn}
The dual perversity of $p$, usually labelled $q$, is the general perversity defined in this way 
\begin{equation}
q=t-p
\end{equation}
where $t$ is the top perversity that is, given a singular stratum $Z$ of $X$, $t(Z)=cod(Z)-2$.
\end{defi}

\begin{exa} The upper middle perversity 
\begin{equation}
 \overline{m}:\{Singular\ Strata\ of\ X\}\rightarrow \mathbb{Z}.
\end{equation}
is defined in the following way: $$\overline{m}(Y)=[\frac{cod(Y)-1}{2}]$$
while the lower middle one is $$t-\overline{m}.$$
\end{exa}

Now we introduce the  notion of $p-$\textbf{allowable} singular simplex :
a singular $i-$simplex in $X$, i.e. a continuous map $\sigma:\Delta_{i}\rightarrow X$, is $p-$\textbf{allowable} if
\begin{equation} \sigma^{-1}(Y)\subset \{(i-cod(Y)+p(Y))-skeleton\ of\ \Delta_{i}\}\ for\ any\ singular\ stratum\ Y\ of\ X.
\end{equation}

A key ingredient in this new theory is the notion of \textbf{homology with stratified  coefficient system}.

\begin{defi} Let $X$ be a stratified pseudomanifold and let $\mathcal{G}$ be a local system on $X-X_{n-1}$. Then the stratified  coefficient system $\mathcal{G}_{0}$ is defined to consist of the pair of coefficient systems given by $\mathcal{G}$ on $X-X_{n-1}$ and the constant $0$ system on $X_{n-1}$ i.e. we think of $\mathcal{G}_{0}$ as consisting of a locally constant fiber bundle $\mathcal{G}_{X-X_{n-1}}$ over $X-X_{n-1}$ with fiber $G$ with the discrete topology together with the trivial bundle on $X_{n-1}$ with the stalk $0.$
\label{sc}
\end{defi}

Then a \textbf{coefficient} $n$ of a singular simplex $\sigma$ can be described by a lift of $\sigma|_{\sigma^{-1}(X-X_{n-1})}$ to $\mathcal{G}$ over $X-X_{n-1}$ together with the trivial lift of $\sigma|_{\sigma^{-1}(X_{n-1})}$ to the $0$ system on $X_{n-1}.$ A coefficient  of a simplex $\sigma$ is considered to be the $0$ coefficient if it maps each  points of $\Delta$ to the $0$ section of one of the coefficient systems. Note that if $\sigma^{-1}(X-X_{n-1})$ is path-connected then a coefficient lift of $\sigma$ to $\mathcal{G}_{0}$ is completely determined by the lift at a single point of $\sigma^{-1}(X-X_{n-1})$ by the lifting extension property for $\mathcal{G}$. The intersection homology chain complex $(I^{p}S_{*}(X,\mathcal{G}_{0}),\partial_{*})$ is defined in the same way as $I^{p}S_{*}(X,G)$, where $G$ is any field, but replacing the coefficient of simplices with coefficient in $\mathcal{G}_{0}$.
 If $n\sigma$ is a simplex $\sigma$ with its coefficient $n$, its boundary is given by the usual formula $\partial(n\sigma)=\sum_{j}(-1)^{j}(n\circ i_{j})(\sigma\circ i_{j})$ where $i_{j}:\Delta_{i-1}\rightarrow \Delta_{i}$ is the $j-$face inclusion map. Here $n\circ i_{j}$ should be interpreted as the restriction of $n$ to the $jth$ face of $\sigma$, restricting the lift to $\mathcal{G}$ where possible and restricting to $0$ otherwise. The basic idea behind the definition is  that when we consider allowability of chains with respect to a perversity, simplices with support entirely in $X_{n-1}$ should vanish and thus not be counted for allowability considerations. We recommend to the reader the references  \cite{GP}, \cite{IGP} and \cite{SP} for a complete development of the subject.\\Finally we conclude this subsection recalling  from \cite{FB} the following definition and the next two theorems:

\begin{defi} Let $X$ be a smoothly stratified pseudomanifold with a Thom-Mather stratification and let $g$ a quasi-edge metric with weights on $reg(X)$. Then the general perversity $p_{g}$ associated with $g$ is:
\begin{equation}p_{g}(Y):= Y\longmapsto [[\frac{l_{Y}}{2}+\frac{1}{2c_{Y}}]]= \left\{
\begin{array}{lll}
0 & l_{Y}=0\\
\frac{l_{Y}}{2}+[[\frac{1}{2c_{Y}}]] & l_{Y}\ even\ and\ l_{Y}\neq 0\\
\frac{l_{Y}-1}{2}+[[\frac{1}{2}+\frac{1}{2c_{Y}}]]& l_{Y}\ odd
\end{array}
\right.
\end{equation}
where $l_{Y}=dimL_{Y}$, $c_Y$ is defined in the second point of Definition \ref{zedge} and given any real and positive number $x$, $[[x]]$ is the greatest integer strictly less than $x$.
\label{pim}
\end{defi}

\begin{teo}
\label{ris}
 Let $X$ be a compact and oriented smoothly stratified pseudomanifold of dimension $n$ with a  Thom-Mather stratification  $\mathfrak{X}$. Let $g$ be a quasi-edge metric with  weights on $reg(X)$, see Definition \ref{zedge}. Let $\mathcal{R}_{0}$ be the stratified coefficient system made of  the pair of coefficient systems given by $(X-X_{n-1})\times \mathbb{R}$ over $X-X_{n-1}$ where the fibers $\mathbb{R}$ have the discrete topology  and the constant $0$ system on $X_{n-1}$. Let $p_{g}$ be the general perversity associated with the metric $g$, see Definition \ref{pim}. Then, for all $i=0,...,n$,  the following isomorphisms hold:
\begin{equation}
I^{q_{g}}H^{i}(X,\mathcal{R}_{0})\cong H_{2,max}^{i}(reg(X),g)\cong \mathcal{H}_{abs}^{i}(reg(X),g)
\label{max}
\end{equation}
\begin{equation}
I^{p_{g}}H^{i}(X, \mathcal{R}_{0})\cong H_{2,min}^{i}(reg(X),g)\cong \mathcal{H}_{rel}^{i}(reg(X),g)
\label{min}
\end{equation} 
where $q_{g}$ is the complementary perversity of $p_{g}$, that is, $q_{g}=t-p_{g}$, $t$ is the usual top perversity and $\mathcal{H}^i_{abs/rel}(reg(X),g)$  are the Hodge cohomology groups defined in \ref{nannabobo}. In particular, for all $i=0,...,n$ the groups $$H_{2,max}^{i}(reg(X),g),\ H_{2,min}^{i}(reg(X),g),\ \mathcal{H}_{abs}^{i}(reg(X),g),\ \mathcal{H}_{rel}^{i}(reg(X),g)$$ are all finite dimensional.
\end{teo}

\begin{proof}
See \cite{FB} Theorem 4.
\end{proof}

\begin{teo} Let $X$ be as in the previous theorem. Let $p$ a general perversity in the sense of Friedman on $X$. If $p$ satisfies the following conditions:
\begin{equation}
 \left\{  
\begin{array}{ll}
p\geq \overline{m}\\
p(Y)=0 & if\ cod(Y)=1
\end{array}
\right.
\end{equation}
then there exists $g$, a quasi-edge edge metric with weights on $reg(X)$, such that
\begin{equation}
I^{p}H^{i}(X, \mathcal{R}_{0})\cong H_{2,min}^{i}(reg(X),g)\cong \mathcal{H}_{rel}^{i}(reg(X),g).
\label{minn}
\end{equation} 
Conversely if $p$ satisfies:
\begin{equation}
 \left\{  
\begin{array}{ll}
p\leq \underline{m} \\
p(Y)=-1 & if\ cod(Y)=1
\end{array}
\right.
\end{equation}
then, also in this case, there exists a quasi-edge metric with weights $h$ on $reg(X)$ such that 
\begin{equation}
I^{p}H^{i}(X,\mathcal{R}_{0})\cong H_{2,max}^{i}(reg(X),h)\cong \mathcal{H}_{abs}^{i}(reg(X),h).
\label{maxx}
\end{equation}
\label{new}
\end{teo}

\begin{proof}
See \cite{FB} Theorem 5.
\end{proof}

\subsection{Applications to the intersection cohomology}

Now, after the previous reminder, we are ready to show some applications of the results of the previous sections.

\begin{prop}
 Let $X$ be a compact and oriented smoothly stratified pseudomanifold of dimension $n$ with a  Thom-Mather stratification  $\mathfrak{X}$. Let $g$ be a quasi-edge metric with  weights on $reg(X)$. Then $$H^{i}_{2,m\rightarrow M}(reg(X),g),\ i=0,...,n$$ is a finite sequence of finite dimensional vector spaces with Poincar\'e duality. Moreover Proposition \ref{frengo} and Proposition \ref{gennaro} apply to this kind of riemannian manifolds.
\label{carim}
\label{mamama}
\end{prop}

\begin{proof} By Theorem \ref{ris} we know that both cohomology groups $H^{i}_{2,max/min}(reg(X),g)$ are finite dimensional. This implies that in the following sequence $H^{i}_{2,m\rightarrow M}(reg(X),g),\ i=0,...,n$ each dimensional vector space is finite dimensional. In this way we are in position to apply Theorem \ref{mario}, Proposition \ref{frengo}, Proposition \ref{gennaro} and therefore the thesis follows. 
\end{proof}

Now consider two general perversities $p, q$ such that $p\leq q$. Then the complex associated with $p$ is a subcomplex of that associated with $q$ and therefore the inclusion $i$ induces a map between the intersection cohomology groups $ I^{q}H^{j}(X,\mathcal{R}_{0})$  and $ I^{p}H^{j}(X,\mathcal{R}_{0}) $ that we call $i^*_j$.    In analogy to the previous section we define for each $j=0,...,n$ 
\begin{equation}
\label{annapunda}
I^{q\rightarrow p}H^{j}(X,\mathcal{R}_{0}):=\im( I^{q}H^{j}(X,\mathcal{R}_{0})\stackrel{i^*_j}{\longrightarrow} I^{p}H^{j}(X,\mathcal{R}_{0}))
\end{equation}
and 
\begin{equation}
I^{q\rightarrow p}\chi(X,\mathcal{R}_{0}):=\sum_{i=0}^{n}(-1)^idim(I^{q\rightarrow p}H^{j}(X,\mathcal{R}_{0}))
\end{equation}

Now we are ready to state the following  proposition:
\begin{prop} 
\label{derede}
 Let $X$ be a compact and oriented smoothly stratified pseudomanifold of dimension $n$ with a  Thom-Mather stratification  $\mathfrak{X}$. Let 
$$ p:\{Singular\ Strata\ of\ X\}\rightarrow \mathbb{N}$$
be a general perversity such that $$p\leq \underline{m}\ \text{and}\ p(Y)=-1$$ for each stratum $Y$ of $X$ with $cod(Y)=1$. Then, if we call $q$ its dual perversity,   we have that   
$$I^{q\rightarrow p}H^{j}(X,\mathcal{R}_{0}),\ j=0,...,n$$ is a finite sequence of finite dimensional vector spaces with Poincar\'e duality. Analogously if $$p\geq \overline{m}\ and\ p(Y)=0$$ for each stratum $Y$ of $X$ with $cod(Y)=1$
then, denoting again with $q$ the dual perversity of $p$, we have that  $$I^{p\rightarrow q}H^{j}(X,\mathcal{R}_{0}),\ j=0,...,n$$ is  a finite sequence of finite dimensional vector spaces with Poincar\'e duality.
\end{prop}

\begin{proof} We know  that $p\leq \underline{m}$. This implies  that $t-p\geq t-\underline{m}$ which in turn implies that $q\geq \overline{m}\geq \underline{m}\geq p$ and therefore the sequence \eqref{annapunda} exists. Moreover we know that $p(Y)=-1$ for each stratum $Y$ of $X$ with $cod(Y)=1$.  
This implies that $p$ satisfies the assumptions of Theorem \ref{new} that is there exists a quasi-edge metric $g$ on $reg(X)$ such that $p_{g}=p$. In this way we can use Proposition \ref{mamama} to get the conclusion.\\  In the same way if $p\geq  \overline{m}$  then we get $p\geq q$ . So we can use again Theorem \ref{new} and Proposition \ref{mamama} to get the assertion.
\end{proof}

We have the  following four immediate corollaries:

\begin{cor}
\label{frengoesto}
In the hypothesis of  Proposition \ref{derede}, if $n$ is odd then: 
\begin{equation}
I^{q\rightarrow p}\chi(X,\mathcal{R}_{0})=0
\label{babbon} 
\end{equation}
\end{cor}

\begin{cor}
\label{gennaron}
 In the same  hypothesis of Proposition \ref{derede} suppose  that
\begin{itemize}
\item $ i_{j}^*:I^qH^j(X,\mathcal{R}_{0})\longrightarrow I^pH^j(X,\mathcal{R}_{0})$ is injective,
\end{itemize}
or that
\begin{itemize}
\item  $ i_{j}^*:I^qH^j(X,\mathcal{R}_{0})\longrightarrow I^pH^j(X,\mathcal{R}_{0})$  is surjective.
\end{itemize}
Then 
\begin{equation}
I^qH^j(X,\mathcal{R}_{0}),\  I^pH^j(X,\mathcal{R}_{0})\  j=0,...,n
\end{equation}
are a finite sequences of finite dimensional vector spaces with Poincar\'e duality.
\end{cor}

\begin{cor}
\label{mummio}
In the hypothesis of  Proposition \ref{derede} we have the following inequalities:
\begin{equation}
\label{swhc}
dim( \im(H^{j}_{c}(reg(X))\stackrel{i_j^*}{\longrightarrow}H^{j}_{dR}(reg(X))))\leq dimI^pH^{j}(X,\mathcal{R}_{0})
\end{equation}
\begin{equation}
\label{blasco}
 dim( \im(H^{j}_{c}(reg(X))\stackrel{i_j^*}{\longrightarrow}H^{j}_{dR}(reg(X))))\leq dimI^qH^{j}(X,\mathcal{R}_{0}).
\end{equation}
Moreover if on $reg(X)$ we have that $\im(H^j_{c}(reg(X))\stackrel{i^*_{j}}{\longrightarrow} H^j_{dR}(reg(X)))$ is not trivial for some $j$ then on $X$ $I^pH^j(X,\mathcal{R}_{0})$ and  $I^qH^j(X,\mathcal{R}_{0})$ are  always non-trivial  for each general perversity $p$ such that $p\leq \underline{m}$ or $p\geq \overline{m}$. Finally, if on $reg(X)$ we have that $H^{j}_{c}(reg(X))\stackrel{i_j^*}\rightarrow H^j_{dR}(reg(X))$ is injective,  then we can improve the inequalities \eqref{swhc} and \eqref{blasco} in the following way:
\begin{equation}
\label{swhcl}
dim(H^{j}_{c}(reg(X)))\leq dim(I^pH^{j}(X,\mathcal{R}_{0}))
\end{equation}
\begin{equation}
\label{blascol}
 dim(H^{j}_{c}(reg(X)))\leq dim(I^qH^{j}(X,\mathcal{R}_{0}))
\end{equation}
\begin{equation}
\label{boos}
b_{n-j}(reg(X))\leq dim(I^pH^{n-j}(X,\mathcal{R}_{0}))
\end{equation}
\begin{equation}
\label{roos}
b_{n-j}(reg(X))\leq dim(I^qH^{n-j}(X,\mathcal{R}_{0}))
\end{equation}
\end{cor}

\begin{proof} All the previous inequalities from \eqref{swhc} to \eqref{blascol} are immediate consequences of the previous results. For the last two inequalities we observe that by Poincar\'e duality, we know that $dim(H^j_c(reg(X)))=dim(H^{n-j}_{dR}(reg(X)))=b_{n-j}(reg(X))$.\\ Moreover, from Theorem \ref{mario}, we know that $H^j_{2,m\rightarrow M}(reg(X),g)\cong H^{n-j}_{2,m\rightarrow M}(reg(X),g)$. Therefore using Corollary \ref{loplp} we get $$b_{n-j}(reg(X))\leq dim(H^{n-j}_{2,m\rightarrow M}(reg(X),g))\leq dim(H^{n-j}_{2,max}(reg(X),g))=dim(I^qH^{n-j}(X,\mathcal{R}_{0}))$$
$$b_{n-j}(reg(X))\leq dim(H^{n-j}_{2,m\rightarrow M}(reg(X),g))\leq dim(H^{n-j}_{2,min}(reg(X),g))=dim(I^pH^{n-j}(X,\mathcal{R}_{0}))$$
 and so the statement follows.
\end{proof}

Gluing together some of  the previous results, now we can state the main result of this section. The first part is a \textbf{ Hodge theorem} for $ \im(I^{q_{g}}H^i(X,\mathcal{R}_0)\rightarrow I^{p_{g}}H^i(X,\mathcal{R}_0))$, that is we will show the existence of a self-adjoint extension of $\Delta_{i}:\Omega_c^i(reg(X))\rightarrow \Omega^i_c(reg(X))$ having the nullspace isomorphic to $ \im(I^{q_{g}}H^i(X,\mathcal{R}_0)\rightarrow I^{p_{g}}H^i(X,\mathcal{R}_0))$. In the second part we will show that $(d+\delta)_{ev}$, that is the Gauss-Bonnet operator having as domain the space of the smooth forms of even degree with compact support, admits a Fredholm extension such that its index has a topological meaning.
\begin{teo}
\label{tizios}
In the same hypothesis or Theorem \ref{ris}; Let $\Delta_{\mathfrak{m},i}$ and $(d_{\mathfrak{m}}+d^*_{\mathfrak{m}})_{ev}$ be the operators, as defined respectively in Corollary \ref{giove} and Proposition \ref{rtlrtl}, associated to the riemannian  manifold $(reg(X), g)$. Then    we have the following results:
\begin{equation}
\label{miop}
Ker(\Delta_{\mathfrak{m},i})\cong \im(I^{q_{g}}H^i(X,\mathcal{R}_0)\rightarrow I^{p_{g}}H^i(X,\mathcal{R}_0))
\end{equation}
\begin{equation}
\label{ssss}
\ind((d_{\mathfrak{m}}+d_{\mathfrak{m}}^*)_{ev})=I^{p_{g}\rightarrow q_{g}}\chi(X,\mathcal{R}_0).
\end{equation}
\end{teo}

\begin{proof}
\eqref{miop} follows by Theorem \ref{ris} and Corollary \ref{giove}; analogously \eqref{ssss} follows from Theorem \ref{ris} and from Proposition \ref{rtlrtl}. 
\end{proof}

Now suppose that $dimX=4n$ where $X$ is as in Proposition \ref{derede}. Let $g$ be a quasi-edge metric with weights on $reg(X)$. Then, by Theorem \ref{ris}, it follows that  $(L^{2}\Omega^i(Reg(X),g), d_{max/min,i})$  are Fredholm complexes and so $(reg(X), g)$ admits the $L^2-$signature $\sigma_{2}(reg(X),g)$ as defined in Definition \ref{nnbb}. Moreover, using again Theorem \ref{ris}, we get that in this case   the $L^2-$signature $\sigma_{2}(reg(X),g)$  is just the analytic version of the \textbf{ perverse signature} introduced by Hunsicker in \cite{H} in the case of $depth(X)=1$ and reintroduced in a purely topological way and generalized to any compact topological pseudomanifolds  by Friedman and Hunsicker in \cite{FH}. In other words, if $p_{g}$ is the general perversity of Definition \ref{pim} and $q_{g}$ it is its dual, then 
\begin{equation}
\label{lsign}
\sigma_{2}(reg(X),g)=\sigma_{q_{g}\rightarrow p_{g}}(X)
\end{equation} 
 and we provided an analytic way to construct $\sigma_{q_{g}\rightarrow p_{g}}(X)$ when $X$ is a smoothly stratified pseudomanifold with a Thom-Mather stratification which generalizes the construction given by Hunsicker in \cite{H} in the particular case of $depth(X)=1$. (For the definition of $\sigma_{q_{g}\rightarrow p_{g}}(X)$ see \cite{FH} pag. 15).\\We have the following corollaries:

\begin{cor} Let $X$ be as in Theorem \ref{ris} and let $g$ and $h$ be two quasi-edge metrics with weights on $reg(X)$. If $p_{g}=p_{h}$ then $$\sigma_{2}(reg(X),g)=\sigma_{2}(reg(X),h).$$ 
\label{kkii}
\end{cor}

\begin{proof} It follows immediately from Theorem \ref{ris}.
\end{proof}

\begin{cor} Let $X$ and $X'$  be as in Theorem \ref{ris}. Let $g$ and $h$ be  two  quasi-edge metric with weights respectively on $reg(X)$ and $reg(X')$. Let $f:X\rightarrow X'$ be a   stratum preserving homotopy equivalence which preserves also the orientations of $X$ and $X'$, see \cite{KW} pag 62 for the definition. Suppose that both $p_{g}$ and $p_{h}$ depend only on the codimension of the strata and that $p_{g}=p_{h}$. Then $$\sigma_{2}(reg(X),g)=\sigma_{2}(reg(X'),h).$$ 
\label{kkkii}
\end{cor}

\begin{proof}
As remarked above, by Theorem \ref{ris}, it follows that $\sigma_{2}(reg(X),g)$  is the perverse signature of  Friedman and Hunsicker associated with the general perversities $p_{g}$ and $t-p_{g}$. Analogously   $\sigma_{2}(reg(X'),h)$  is the perverse signature of  Friedman and Hunsicker associated with the general perversities $p_{h}$ and $t-p_{h}$. So the statement follows by the invariance of the perverse signature under  the action of stratum preserving homotopy equivalences which preserve also the orientations.
\end{proof}

\subsection{Some examples of manifolds without riemannian metric with finite $L^2-$cohomology}

Now we go ahead  showing an example of a manifold $M$ such that $\im(H_{c}^i(M)\rightarrow H^i_{dR}(M))$ is  infinite dimensional. To do this we start with the following definition:
\begin{defi}
\label{ends}
 Let $M$ be a smooth manifold and let $A\subset M$.
\begin{enumerate}
\item We will say that $A$ is bounded if its closure, $\overline{A}$, is compact.
\item We will say that $M$ has only one end if for each compact subset $K\subset M$ $M-K$ has only one unbounded connected component.
\item We will say that M has k ends (where $k \geq 2$) if there is a compact
set $K_{0} \subset M$ such that for every compact set $K \subset M$ containing $K_0$, $ M - K$ has
exactly $k$ unbounded connected components.
\end{enumerate} 
\end{defi}

The following proposition is a modified version of Lemma 2.3 in \cite{GC}:

\begin{prop} 
\label{inj}
Let $M$ be a manifold with only one end. Then the natural map $$H^1_{c}(M)\rightarrow H^1_{dR}(M)$$ is injective.
\end{prop}

\begin{proof}
Let $\alpha\in \Omega^{1}_{c}(M)$ be closed and let $f:M \rightarrow \mathbb{R}$ be a smooth function such that $df=\alpha$.  This implies the existence of a constant $c$ such that $f|_{M-supp(\alpha)}=c$. Therefore, by the fact that $M$ has only one end, we get that $f-c$ has compact support.
\end{proof}

Now using Poincar\'e duality for open and oriented manifolds we know that the de Rham cohomology with compact support is infinite dimensional if and only if the de Rham cohomology is infinite dimensional. This implies that if $M$ is a smooth and oriented surface with only one end and such that $H^{1}_{dR}(M)$ is infinite dimensional then also $\im(H_{c}^1(M)\rightarrow H^1_{dR}(M))$ is infinite dimensional. So we can state the following proposition:

\begin{prop}
\label{lastlast}
Let $M$ be an open and oriented surface with infinite genus and with only one end. Then  $\im(H_{c}^1(M)\rightarrow H^1_{dR}(M))$ is infinite dimensional and therefore on $M$, according to Corollary \ref{gianni}, there is no  riemannian metric g (complete or incomplete) such that,  for some closed extension $(L^{2}\Omega^{*}(M,g), D_{*})$ of $(\Omega^{*}_{c}(M),d_{*})$,  one  of the following properties is satisfied:
\begin{enumerate}
\item  $\overline{H}^{1}_{2,D_{*}}(M,g)$  is finite dimensional.
\item  $H^{1}_{2,D_{*}}(M,g)$  is finite dimensional.
\item $D_{1}^*\circ D_{1}+D_{0}\circ D_{0}^*$ on its domain (as defined in \eqref{saed}) endowed with the graph norm is a Fredholm operator.
\end{enumerate}
Moreover on $M$ there is no  riemannian metric $g$ such that:
\begin{enumerate}
\item $\Delta_{max,1}$, the maximal closed extension of $\Delta_{j}:\Omega_{c}^1(M)\rightarrow \Omega^1_{c}(M)$,  has finite dimensional nullspace.
\item  $\Delta_{min,1}$, the minimal closed extension of $\Delta_{1}:\Omega_{c}^j(M)\rightarrow \Omega^1_{c}(M)$,  satisfies\\ $dim(ran(\Delta_{min,1})^{\bot})<\infty$.
\end{enumerate}
\label{asso}
\end{prop}

The rest of this subsection is devoted to show another example of an open manifold which satisfies Corollary \ref{gianni} but that it is not contemplate in the previous proposition. To do this  we  state the following lemma which gives a sufficient condition to have $\im(H^{n-1}_{c}(M)\rightarrow H^{n-1}_{dR}(M))$ infinite dimensional where $n=dim(M).$

\begin{lemma}
\label{tyhg}
 Let $M$ be an open and oriented smooth manifold of dimension $n$. Assume that there exists a sequence of open subsets $\{A_{j}\}_{j\in J}$  such that:
\begin{enumerate}
\item $\partial\overline{A_{j}}$ is smooth for each $j$.
\item Every connected component of $M-A_j$ has connected boundary.
\item $\lim_{j\rightarrow \infty}dim(\im(H_{c}^1(A_{j})\rightarrow H^1_{dR}(A_{j})))=\infty$.
\end{enumerate}
Then  $ \im(H_{c}^{n-1}(M)\rightarrow H^{n-1}_{dR}(M))$ is infinite dimensional.
\end{lemma}

\begin{proof}
It is an immediate consequence of the next proposition.
\end{proof}

\begin{prop} Let $M$ be an open and oriented smooth manifold of dimension $n$.  Assume that there exists an open subset $A\subset M$ such that every connected component of $M-A$ has connected boundary. Then there is a natural and injective map  $$\im (H^{1}_{c}(A)\rightarrow H^{1}_{dR}(A))\longrightarrow( \im(H^{n-1}_{c}(M)\rightarrow H^{n-1}_{dR}(M)))^*.$$ 
\end{prop}

\begin{proof} Consider the following pairing:
\begin{equation}
\im(H^{1}_{c}(A)\rightarrow H^{1}_{dR}(A))\times \im(H^{n-1}_{c}(M)\rightarrow H^{n-1}_{dR}(M)) \longrightarrow \mathbb{R},\  ([\omega],[\eta])\mapsto \int_{M} \omega\wedge \eta
\label{camborata}
\end{equation}
where $\omega$ is a closed $(n-1)-$form with compact support in $A$  and   $\eta$ is a closed $1-$form with compact support in $M$. First of all we have to show that \eqref{camborata} is well defined. As observed at the end of subsection 2.3 
a cohomology class in $\im(H^i_{c}(M)\rightarrow H^i_{dR}(M))$, or in $\im(H^i_{c}(A)\rightarrow H^i_{dR}(A))$, is just a cohomology class in $H^i_{dR}(M)$, or in $H^i_{dR}(A)$,  such that it admits a representative with compact support respectively in $M$  or $A$. Now let $\omega, \omega'$ be two closed $1-$forms with compact support in $A$ such that $[\omega]=[\omega']$ in $\im(H^{1}_{c}(A)\rightarrow H^{1}_{dR}(A))$. Then for $\omega=\omega'+df$ for some $f\in C^{\infty}(A,\mathbb{R})$. But $df$ has compact support contained  in $A$ and every connected component of $M-A$ has connected boundary. Therefore there exists $f'\in C^{\infty}(M)$ such that $f'|_A=f$ and $d(f'|_{(M-A)})=0$. Finally let $\eta, \eta'$ be two closed $(n-1)$-forms such that $[\eta]=[\eta']$ in  $\im(H^{n-1}_{c}(M)\rightarrow H^{n-1}_{dR}(M))$ and let $\psi\in \Omega^{n-2}(M)$ be such that $\eta=\eta'+d\psi$. Then:
$$\int_M\omega\wedge\eta=\int_A\omega\wedge\eta=\int_A(\omega'+df)\wedge(\eta'+d\psi)=
\int_A(\omega'+df')\wedge(\eta'+d\psi)=$$ $$=\int_M(\omega'+df')\wedge(\eta'+d\psi)=(\text{by Stokes Theorem})\ \int_M\omega'\wedge\eta'$$ and therefore \eqref{camborata} is well defined.F
Now let $[\omega]\in \im(H^{1}_{c}(A)\rightarrow H^{1}_{dR}(A))$ such that for each class $[\eta]\in \im(H^{n-1}_{c}(M)\rightarrow H^{n-1}_{dR}(M))$ the pairing \eqref{camborata} is zero. This implies that for each smooth and closed $(n-1)-$form $\phi$ with compact support in $M$ we have $$\int_{M}\phi\wedge \omega=0.$$In particular this is true for  each smooth and closed $(n-1)-$form $\phi$ with compact support in $A$ and therefore,  using again the Poincar\'e duality for open and oriented manifold, we get that there exists $\beta \in C^{\infty}(A)$ such that $d\beta=\omega$. So we can conclude that $[\omega]=0$ in $\im(H^1_{c}(A)\rightarrow H^1_{dR}(A))$ and therefore from the pairing \eqref{camborata} we get the desired injective map.
\end{proof}

Using the previous lemma we have the following  corollary that was suggested to the author by Pierre Albin:

\begin{cor}
\label{lastlastlast}
Let $M$ be an open and oriented surface obtained by gluing an infinite but countable  family of tori. Suppose that $M$ has a  finite number of  ends. Assume moreover that there exists an exhausting sequence of  open subsets with compact closure, $\{A_{j}\}_{j\in \mathbb{N}}$,  which satisfies the following properties:
\begin{enumerate}
\item $M-A_{j}$ is disconnected, made of $k$ unbounded connected components, where $k$ is the number of ends of $M$.
\item $\partial A_{j}$ is smooth and made of $k$  connected components.  
\end{enumerate} 
Then  $\im(H_{c}^{1}(M)\rightarrow H^{1}_{dR}(M))$ is infinite dimensional  and therefore on $M$, according to Corollary \ref{gianni}, there is no  riemannian metric g (complete or incomplete) such that,  for some closed extension $(L^{2}\Omega^{*}(M,g), D_{*})$ of $(\Omega^{*}_{c}(M),d_{*})$,  one  of the properties listed in Prop. \ref{lastlast} is satisfied.
\end{cor}

\begin{proof}
First of all we remark that, given an infinite but countable sequence of tori, it is immediate to check that it is possible to glue them together obtaining a surface which satisfies the assumptions of the corollary. The goal now is to show that we can apply Lemma \ref{tyhg}. We start observing  that, for every $j\in \mathbb{N}$, each of the connected components of $A_j$ is a compact smooth one dimensional manifold and therefore it is diffeomorphic to $S^1$. Now let us label by $\Sigma_j$  the closed and oriented surface obtained by gluing to $A_{j}$ $k$ copies of $\overline{\mathbb{B}}$, the unit ball in $\mathbb{R}^2$ with boundary. Clearly, if we label with $g(\Sigma_j)$ the genus of $\Sigma_j$ then we have 
\begin{equation}
\label{pallade}
\lim_{j\rightarrow \infty}g(\Sigma_j)=\infty
\end{equation}
Now, recalling that $2-2g(\Sigma_j)=\chi(\Sigma_j)=b_0(\Sigma_j)-b_1(\Sigma_j)+b_2(\Sigma_j)=2-b_1(\Sigma_j)$ and using the Mayer-Vietoris sequence it is not hard to see that $dim(H^{1}(A_{j}))\geq 2g(\Sigma_j)-k $ where $k$ is the number of ends of $M$ and therefore it is fixed. Therefore the sequence $\{A_j\}$  satisfies:
\begin{equation}
\label{loveorca}
\lim_{j\rightarrow \infty}dim(H^1_{dR}(A_{j}))=\infty.
\end{equation}
Now  we recall the fact that, on a compact and oriented manifold with boundary $\overline{M}$, we have $H^i(\overline{M},\partial \overline{M})\cong H^i_c(M)$ and $H^i_{dR}(\overline{M})\cong H^i_{dR}(M)$ where $M$ is the interior of $\overline{M}$. So, from the long exact sequence for the relative de Rham cohomology on a compact manifold with boundary, it is easy to show that $dim(H^1(A_{j}))=dim(\im(H^1_{c}(A_{j})\rightarrow H^{1}_{dR}(A_{j})))+\lambda_{A_{j}}$ where $\lambda_{A_{j}}\in \{0,...,k\}$ is defined as the dimension of the image $\im(H^1_{dR}(\overline{A_j})\rightarrow H^1_{dR}(\partial\overline{A_j}))$. This means that the correction term $\lambda_{A_{j}}$ could depend on $A_{j}$ but in any case it lies in $\{0,...,k\}$ which is a bounded set being $k$ fixed. Therefore, from this equality and from \eqref{loveorca}, we get that 
 $$ \lim_{j\rightarrow \infty}dim(\im(H^1_{c}(A_{j})\rightarrow H^{1}_{dR}(A_{j})))=\infty.$$  This implies that we can apply Lemma \ref{tyhg} and therefore the statement follows.
\end{proof}

Finally, using the notions introduced in Definition \ref{ends} and Proposition \ref{inj}, we conclude the section giving another application to the stratified pseudomanifolds and intersection cohomology.
\begin{prop}
\label{qzzq}
Let $X$ be as in Theorem \ref{ris}. Suppose that $X$ is normal, that is for each $p\in sing(X)$ there exists an open neighborhood $U$ such that $U-(U\cap sing(X))$ is connected. Then, if $sing(X)$ is connected, $reg(X)$ is an open manifold with only one end.
\end{prop}

\begin{proof}
Let $K\subset reg(X)$ a compact subset. If $reg(X)-K$ is connected then we have nothing to show. Suppose therefore that it is disconnected and let $A_{1},...,A_{l}$ be the connected components. By the fact that $X$ is normal we get that there exists an open neighborhood $sing(X)\subset V\subset X$ such that $V-sing(X)$ is connected.  By the fact that $K\subset reg(X)$ we get that we can choose $V$ such that $V\cap K=\emptyset$. Therefore we get $V=\cup_{i=1}^l(\overline{A}_{i}\cap V)$ and  this equality implies that $V- sing(X)=\cup_{i=1}^l(A_{i}\cap (V-sing(X))).$ Every subset $A_{i}\cap (V-sing(X))$ is an open subset of $V-sing(X)$ and for each $i,j\in \{1,...,l\}, i\neq j$ we have $(A_{i}\cap(V-sing(X)))\cap(A_{j}\cap(V-sing(X)))=\emptyset.$ So the fact that $V-sing(X)$ is connected, joined with the fact that  $V- sing(X)=\cup_{i=1}^l(A_{i}\cap (V-sing(X)))$,  implies that there exists just one index in $\{1,...,l\}$, which we label $\gamma$, such that $A_{\gamma}\cap(V-sing(X))\neq \emptyset$. So we can conclude that: 
\begin{enumerate}
\item $V-sing(X)\subset A_{\gamma}$.
\item $A_{\gamma}\cup sing(X)=A_{\gamma}\cup V$ is open in $X$.
\end{enumerate}
This implies that if we label $\mathfrak{K}$ the closure in  $X$ of $$(\bigcup_{i=1,i\neq \gamma}^lA_{i})\cup K$$ then we have \begin{equation}
\label{materia}
\mathfrak{K}\subseteq  X-(A_{\gamma}\cup sing(X))
\end{equation}
and therefore   $\mathfrak{K}$ is a compact subset of $X$. But  \eqref{materia} implies that $\mathfrak{K}\subset reg(X)$ and therefore it is a compact subset of $reg(X)$. This allows us to conclude that for each $i\in \{1,...,l\},\ i\neq \gamma$ we have that $\overline{A_{i}}$ is a compact subset of $reg(X)$ and so we got the statement.
\end{proof}

We have the following corollary:
\begin{cor}
\label{falco}
Let $X$ be as in Theorem \ref{ris} such that $X$ is normal and $sing(X)$ is connected. Let $p$ be a general perversity as in the statement of Theorem \ref{new} and let $q$ be its dual. Then we have the following inequalities:
\begin{enumerate}
\item $dim(H^1_{c}(reg(X)))\leq I^pb_{1}(X,\mathcal{R}_{0}),\  dim(H^1_{c}(reg(X)))\leq I^qb_{1}(X, \mathcal{R}_{0}).$
\item $b_{n-1}(reg(X))\leq I^pb_{n-1}(X,\mathcal{R}_{0}),\ b_{n-1}(reg(X))\leq I^qb_{n-1}(X,\mathcal{R}_{0}).$
\end{enumerate}
where $I^pb_{i}(X,\mathcal{R}_{0})$ is the dimension of $I^pH^i(X,\mathcal{R}_{0})$ and analogously $I^qb_{i}(X,\mathcal{R}_{0})$ is the dimension of $I^qH^i(X,\mathcal{R}_{0})$. Finally if $dimX=2$ and $cod(sing(X))=0$ then 
\begin{equation}
\label{fine}
I^{\underline{m}}\chi(X)\leq \chi(reg(X))+1
\end{equation}
where $I^{\underline{m}}\chi(X)=\sum_{i=0}^2(-1)^iI^{\underline{m}}b_{i}(X)$. 
\end{cor}

\begin{proof}
From Proposition \ref{qzzq} we know that $reg(X)$ has only one end. Therefore from Proposition \ref{ends} we get that the map $H^1_c(M)\rightarrow H^1_{dR}(M)$ is injective and so the thesis follows by Corollary \ref{mummio}. Before to prove the second part of the corollary we do the following observation: by the assumption we know that $H^1_{c}(reg(X))$ is finite dimensional; using Poincar\'e duality for open and oriented manifolds this implies that $b_{i}(reg(X))$ is finite dimensional for each $i=0,...,2$ and therefore $\chi(reg(X))$ makes sense.  Now the assumptions on $X$ imply that $sing(X)=\{p\}$ and $X$ is a Witt space (for the definition of Witt space see for example \cite{KW} pag 75). It is well known that, over a Witt space, the intersection cohomology associated with the lower middle perversity satisfies Poincar\'e duality, that is we have $I^{\underline{m}}H^i(X)\cong I^{\underline{m}}H^{2-i}(X)$. Poincar\'e duality for open and oriented manifolds implies that $b_2(reg(X))=dim(H^0_{c}(reg(X)))=0$. So, using the previous statements of this corollary, we have $I^{\underline{m}}\chi(X)=2-I^{\underline{m}}b_{1}(X)\leq 2-b_{1}(reg(X))\leq 1+1-b_{1}(reg(X))=1+\chi(reg(X))$.
\end{proof}

\subsection{Some applications to the Friedrichs extension of $\Delta_{i}$}

This last section is devoted to show some properties of the Friedrichs extension of $\Delta_i:\Omega_c^i(M)\rightarrow \Omega_c^i(M)$.\\ The main result is to show that if $(M,g)$ is an open and oriented riemannian manifold such that $(L^{2}\Omega^*(M,g),d_{max/min,*})$ are Fredholm complexes then, for each $i=0,...,dimM$,  the Friedrichs extension of $\Delta_{i}:\Omega_{c}^i(M)\rightarrow \Omega^{i}_c(M)$ is a Fredholm operator. In particular this applies when $M$ is the regular part of a compact and smoothly stratified pseudomanifold with a Thom-Mather stratification and $g$ is a quasi-edge metric with weights on $reg(X)$. We start recalling the definition of the Friedrichs extension:

\begin{defi}
\label{tgbt}
 Let $H$ be a Hilbert space and let $B:H\rightarrow H$ be a densely defined operator. Suppose that $B$ is positive, that is for each $u\in \mathcal{D}(B)$ we have $<Bu,u>\geq 0$. The Friedrichs extension of $B$, usually labeled $B^{\mathcal{F}}$, is the operator defined in the following way: $$\mathcal{D}(B^{\mathcal{F}})=\{u\in \mathcal{D}(B^*):\ there\ exists\ \{u_{n}\}\subset \mathcal{D}(B)\  such\ that\  <u-u_n,u-u_n>\rightarrow 0\  and$$ $$ <B(u_n-u_m),u_n-u_m>\rightarrow 0\ for\ n,m\rightarrow \infty\}\ and\ we\ put\ B^{\mathcal{F}}(u)=B^*(u).$$
\end{defi}

\begin{prop} In the same assumptions of the previous definition $B^{\mathcal{F}}$ is a positive self-adjoint extension of $B$.
\label{kkjjkk}
\end{prop}

\begin{proof}
See \cite{MM} appendix C. 
\end{proof}

\begin{lemma}
\label{rerrez}
 Let $A_j:H_j\rightarrow H_j$, $j=1,2,$  be  two positive and densely defined operators. Then on $H_1\oplus H_2$, with the natural Hilbert space structure of a direct sum, we have: $$(A_1\oplus A_2)^{\mathcal{F}}=A_1^{\mathcal{F}}\oplus A_2^{\mathcal{F}}.$$
\end{lemma}

\begin{proof} The assumptions of the lemma imply that $A_1\oplus A_1:H_1\oplus H_2\rightarrow H_1\oplus H_2$ is densely defined and positive. Moreover it is clear that $(A_1\oplus A_2)^*=A_1^*\oplus A_2^*$.\\ Now let $(a,b)\in \mathcal{D}((A_1\oplus A_2)^{\mathcal{F}})$. From Definition \ref{tgbt} we get that $(a,b)\in \mathcal{D}((A_1\oplus A_2)^*)$ and there exists a sequence $\{(a_{n},b_{n})\}\subset \mathcal{D}(A_1\oplus A_2)$ such that: $$(a_{n},b_{n})\rightarrow (a,b)\ \text{and}\ <A\oplus B((a_{n},b_{n})-(a_{m},b_{m})),(a_{n},b_{n})-(a_{m},b_{m})>\rightarrow 0.$$ 
Furthermore from the same definition we know that $(A_1\oplus A_2)^\mathcal{F}(a,b)=(A_1\oplus A_2)^*(a,b)$. But from these requirements we get immediately that $a\in \mathcal{D}(A_1^*),\ b\in \mathcal{D}(A_2^*)$, $\{a_{n}\}\subset \mathcal{D}(A_1)$, $\{b_{n}\}\subset \mathcal{D}(A_2)$, $a_{n}\rightarrow a$, $<A_1(a_{n}-a_{m}),a_{n}-a_{m}>\rightarrow 0$ and analogously that $b_{n}\rightarrow b$ and that $<A_2(b_{n}-b_{m}),b_{n}-b_{m}>\rightarrow 0$. Therefore this imply that $a\in \mathcal{D}(A_1^\mathcal{F})$, $b\in \mathcal{D}(A_2^\mathcal{F})$ and $(A_1\oplus A_2)^{\mathcal{F}}(a,b)=A_1^{\mathcal{F}}(a)\oplus A_2^{\mathcal{F}}(b)$. In this way we know that $A_1^{\mathcal{F}}\oplus A_2^{\mathcal{F}}$ is an extension of $(A_1\oplus A_2)^{\mathcal{F}}$. Moreover it  is clear that also $A_1^{\mathcal{F}}\oplus A_2^{\mathcal{F}}$  is a self-adjoint operator because it is a direct sum of two self-adjoint operators acting on $H_1$ and $H_2$ respectively.  Finally, by the fact that  both $A_1^{\mathcal{F}}\oplus A_2^{\mathcal{F}}$ and $(A_1\oplus A_2)^{\mathcal{F}}$ are self-adjoint operators, it follows that $A_1^{\mathcal{F}}\oplus A_2^{\mathcal{F}}=(A_1\oplus A_2)^{\mathcal{F}}$.
\end{proof}

\begin{rem}
\label{halifax}
It is clear that the previous proposition generalizes to the case of a finite sum, that is if we have $A_j:H_j\rightarrow H_j$ $j=1,...,n$ such that for each $j$ $A_j$ is positive and densely defined then:
 \begin{equation}
\label{toporobot}
(A_1\oplus...\oplus A_n)^{\mathcal{F}}:\bigoplus_{j=1}^n H_j\rightarrow \bigoplus_{j=1}^n H_j=A_1^{\mathcal{F}}\oplus...\oplus A_n^{\mathcal{F}}:\bigoplus_{j=1}^n H_j\rightarrow \bigoplus_{j=1}^nH_j
\end{equation}
\end{rem}

\begin{lemma}
\label{weewq}
 Let $E,F$ be two vector bundles over an open, incomplete and oriented riemannian manifold $(M,g).$ Let $g$ and $h$ be two metrics on $E$ and $F$ respectively. Let $d:C^{\infty}_{c}(M,E)\rightarrow C^{\infty}_{c}(M,F)$ be an unbounded and densely defined differential operator. Let $d^t:C^{\infty}_{c}(M,F)\rightarrow C^{\infty}_{c}(M,E)$ be its formal adjoint. Then for  $d^t\circ d:L^{2}(M,E)\rightarrow L^{2}(M,E)$ we have: $$(d^t\circ d)^{\mathcal{F}}=d_{max}\circ d_{min}.$$
\end{lemma}

\begin{proof}
See \cite{BrL}, lemma 3.1 pag. 447.
\end{proof}

From lemma \ref{weewq} we get, as it is showed in \cite{BrL} pag. 448, the following useful corollary:

\begin{cor}
\label{corvo}
Let $(M,g)$ be an open, oriented and incomplete riemannian manifold of dimension $n$. Consider the Laplacian acting on the space of smooth forms with compact support: $$\Delta:\bigoplus_{i=0}^n\Omega_{c}^i(M)\longrightarrow \bigoplus_{i=0}^n\Omega_{c}^i(M).$$ Then for $$\Delta^\mathcal{F}:\bigoplus_{i=0}^n L^{2}\Omega^i(M,g)\longrightarrow \bigoplus_{i=0}^nL^2\Omega^i(M,g)$$ we have $$\Delta^\mathcal{F}=(d+\delta)_{max}\circ (d+\delta)_{min}.$$
\end{cor}

Now we are in positions to state the following result:

\begin{teo}
\label{deddf}
Let $(M,g)$ be an open, oriented and incomplete riemannian manifold of dimension $n$. Then for each $i=0,...,n$ we have the following properties:
\begin{enumerate}
\item If $\overline{H}^i_{2,m\rightarrow M}(M,g)$ is finite dimensional then $Ker(\Delta^\mathcal{F}_i)$ and $\mathcal{H}^i_{min}(M,g)$ are finite dimensional.
\item If $(L^{2}\Omega^*(M,g), d_{max,*})$ is a Fredholm complex, or equivalently if $(L^{2}\Omega^*(M,g), d_{min,*})$ is a Fredholm complex, then for each $i$ $\Delta^{\mathcal{F}}_{i}$ is a Fredholm operator on its domain endowed with graph norm. 
\end{enumerate}
\end{teo}

\begin{proof}
Consider $\pi_{abs,i}:L^{2}\Omega^i(M,g)\rightarrow \mathcal{H}^i_{abs}(M,g)$ that is the projection on $\mathcal{H}^{i}_{abs}(M,g)$.  We know that $\pi_{abs,i}(\mathcal{H}^i_{rel})\cong \overline{H}^{i}_{2,m\rightarrow M}(M,g)$. This property is shown in a more general context in the proof of Theorem \ref{duality} and remarked in Remark \ref{piop}. But $\mathcal{H}^{i}_{min}(M,g)=Ker(d_{min,i})\cap Ker(\delta_{min,i-1})=\mathcal{H}^i_{abs}(M,g)\cap \mathcal{H}^i_{rel}(M,g)$. So $\mathcal{H}^i_{min}(M,g)\subseteq \pi_{abs,i}(\mathcal{H}^{i}_{rel}(M,g))$ and therefore the  statement follows. Now, by Lemma \ref{rerrez} and Remark \ref{halifax}, we know that  $\Delta^{\mathcal{F}}=\bigoplus_i\Delta_{i}^{\mathcal{F}}$ which in particular implies that $Ker(\Delta^{\mathcal{F}})=\bigoplus_iKer(\Delta_{i}^{\mathcal{F}})$. But from Corollary \ref{corvo} we get $$Ker(\Delta^{\mathcal{F}})=Ker((d+\delta)_{min})\subseteq \bigoplus_{i=0}^n(Ker(d_{min,i})\cap Ker(\delta_{min,i-1}))=\bigoplus_{i=0}^n \mathcal{H}^i_{min}(M,g).$$ Therefore we can conclude that $Ker(\Delta_i^{\mathcal{F}})$ is finite dimensional for each $i=0,...,n$.\\ Now consider the second point; we want to show that if $(L^{2}\Omega^{*}(M,g),d_{max,*})$ is a  Fredholm complex then also $$(d+\delta)_{max}\circ (d+\delta)_{min}:\bigoplus_{i=0}^nL^2\Omega^{i}(M,g)\rightarrow \bigoplus_{i=0}^nL^2\Omega^{i}(M,g)$$ is a Fredholm operator. By the previous point, we already know that the nullspace of $(d+\delta)_{max}\circ (d+\delta)_{min}$ is finite dimensional. So we have to show that its range  is closed with finite dimensional orthogonal complement. To do this is equivalent to show that the cokernel of $(d+\delta)_{max}\circ (d+\delta)_{min}$ is finite dimensional. We will do this by showing  that $ran((d+\delta)_{max}\circ (d+\delta)_{min}) = ran((d+\delta)_{max})$ and that $(d+\delta)_{max}$ has finite dimensional cokernel. By the fact that $(d+\delta)_{min}^*=(d+\delta)_{max}$ it follows that 
\begin{equation}
\label{domain}
ran((d+\delta)_{max})=\{(d+\delta)_{max}(u):\ u\in \overline{ran((d+\delta)_{min})}\cap \mathcal{D}((d+\delta)_{max}) \}.
\end{equation}
Now it is  easy to check that if $(L^{2}\Omega^{*}(M,g),d_{max,*})$ is a Fredholm complex then  $d_{max}+\delta_{min}$ is  a Fredholm operator on its domain endowed with the graph norm. But  the fact that  $ran(d_{max}+\delta_{min})\subset ran((d+\delta)_{max})$ implies that there is a surjective map $$\frac{(\bigoplus_{i=0}^nL^{2}\Omega^i(M,g))}{ran((d+\delta)_{max})}\longrightarrow  \frac{(\bigoplus_{i=0}^nL^{2}\Omega^i(M,g))}{ran(d_{max}+\delta_{min})}.$$ So $(d+\delta)_{max}$ on its domain with the graph norm is a bounded linear operator with finite dimensional cokernel and this implies that the range of $(d+\delta)_{max}$ is closed with finite dimensional orthogonal complement.
 But $((d+\delta)_{max})^*=(d+\delta)_{min}$ and therefore also $(d+\delta)_{min}$ has closed range. In this way \eqref{domain} becomes
\begin{equation}
\label{domainn}
ran((d+\delta)_{max})=\{(d+\delta)_{max}(u):\ u\in ran((d+\delta)_{min})\cap \mathcal{D}((d+\delta)_{max}) \}.
\end{equation} 
So we can conclude that $ran((d+\delta)_{max}\circ (d+\delta)_{min}) = ran((d+\delta)_{max})$ and therefore $(d+\delta)_{max}\circ (d+\delta)_{min}$ is a Fredholm operator.\\
Now,  by the equality $(d+\delta)_{max}\circ (d+\delta)_{min}=\bigoplus_{i=0}^n\Delta_{i}^{\mathcal{F}}$, we get,  for each $i=0,...,n$, that also $\Delta_{i}^\mathcal{F}$ has closed range. Moreover we already know that its nullspace of  $\Delta_{i}^\mathcal{F}$  is finite dimensional and so, because it is self-adjoint and with closed range, we can conclude that it is Fredholm. This completes the proof.
\end{proof}

As mentioned at the beginning of the section  the following corollary is an application of the previous theorem; it already known when $X$ is a compact manifold with isolated singularities for any positive conic operator (see \cite{MAL}) and also for $\Delta_i^{\mathcal{F}}$ when $(M,g)$ is a manifold with incomplete edges, see \cite{MAV}.

\begin{cor}
\label{ilboss}
Let $X$ be a compact smoothly and oriented stratified pseudomanifold of dimension $n$ with a Thom Mather stratification. Let $g$ be a quasi-edge metric with weights on $reg(X)$. Then on $L^{2}\Omega^i(reg(X),g)$, for each $i=0,...,n$, $\Delta_{i}^\mathcal{F}$ is a Fredholm operator on its domain endowed with the graph norm.
\end{cor}

\section{Final considerations}
Consider again an open, oriented and incomplete  riemannian manifold $(M,g)$ of dimension $n$.  By Corollary \ref{loplp} we now that that there is a copy of $\im(\overline{H}^i_{2,min}(M,g)\rightarrow\overline{H}^{i}_{2,max}(M,g))$ in each $i-th$ reduced cohomology group $\overline{H}^i_{2,D_*}(M,g)$ of each closed extension $(L^2\Omega^*(M,g),D_*)$ of $(\Omega_c^*(M),d_*).$ In the same way, using again Corollary \ref{loplp}, we know that there is a copy of $\im(H^i_{2,min}(M,g)\rightarrow H^i_{2,max}(M,g))$ in each $i-th$ reduced cohomology group $H^i_{2,D_*}(M,g)$ of each closed extension $(L^2\Omega^*(M,g),D_*)$ of $(\Omega_c^*(M)^,d_*).$  In particular, by Theorem \ref{polcas}, we know that when $d_{min,i}$ has closed range for each $i$ then the groups $\im(H^i_{2,min}(M,g)\rightarrow Hi_{2,max}(M,g))$ are really the cohomology groups of an Hilbert complex that we labeled $(L^2\Omega^i(M,g),d_{\mathfrak{m},i})$. Therefore we can look at $\im(H^i_{2,min}(M,g)\rightarrow H^i_{2,max}(M,g))$ as the \textbf{smallest possible} $L^2-$cohomology groups for $(M,g)$.\\ From the Hodge point of view the smallest Hodge cohomology groups are $\mathcal{H}_{min}^{i}(M,g)$ defined, for each $i=0,...,n$,  as $Ker(d_{min,i})\cap Ker(\delta_{min,i-1})$.
Therefore a natural question is:
\begin{itemize}
\item Is there any relation between $\mathcal{H}_{min}^{i}(M,g)$ and $\im(\overline{H}^i_{2,min}(M,g)\rightarrow \overline{H}^{i}_{2,max}(M,g))$ or between 
$\mathcal{H}_{min}^{i}(M,g)$ and $\im(H^i_{2,min}(M,g)\rightarrow H^{i}_{2,max}(M,g))$?
\end{itemize}
In \cite{HM} Theorem 4.8, using techniques arising from Mazzeo's edge calculus, the authors showed that if $(M,g)$ is an incomplete manifold with edge then we have the following isomorphism:
\begin{equation}
\label{itsmywarrior}
\mathcal{H}^i_{min}(M,g)\cong \im(H^i_{2,min}(M,g)\rightarrow H^{i}_{2,max}(M,g)).
\end{equation}
Therefore, using Corollary \ref{giove},  we get the following immediate consequences:
\begin{cor}
\label{aswqas}
Let $(M,g)$ be an incomplete manifold with edge. Then, for each $i=0,...,n$
\begin{enumerate}
\item $Ker(\Delta_{\mathfrak{m},i})=\mathcal{H}^i_{min}(M,g)$
\item $ran(\Delta_{\mathfrak{m},i})=\overline{ran(d_{max,i-1})+ran(\delta_{max,i}).}$
\end{enumerate}
\end{cor}

Finally we conclude the section showing that the isomorphism \eqref{itsmywarrior} is equivalent to requirement that the Hilbert space $L^2\Omega^i(M,g)$ satisfy some geometric properties.
\begin{prop}
\label{dsdzm}
 Let $(M,g)$ be an open oriented and incomplete riemannian manifold. Suppose that, for each $i=0,...,n$, $\im(\overline{H}^i_{2,min}(M,g)\rightarrow \overline{H}^{i}_{2,max}(M,g))$ is finite dimensional.
Then there exists alway an injective map 
\begin{equation}
\label{cesemoquasi}
\mathcal{H}^i_{min}(M,g)\rightarrow \im(\overline{H}^i_{2,min}(M,g)\rightarrow \overline{H}^{i}_{2,max}(M,g)).
\end{equation}
 Moreover the following properties are equivalent:
\begin{enumerate}
\item $\mathcal{H}_{min}^{i}(M,g)\cong \im(\overline{H}^i_{2,min}(M,g)\rightarrow \overline{H}^{i}_{2,max}(M,g))$
\item $\mathcal{H}^i_{abs}(M,g)=\mathcal{H}_{min}^i(M,g)\oplus(\overline{ran(\delta_{max,i})}\cap \mathcal{H}^i_{abs}(M,g))$
\item Let $\pi_{abs/rel/min,i}:L^2\Omega^i(M,g)\rightarrow \mathcal{H}_{ abs/rel/min}^i(M,g)$ be the orthogonal projections  of $L^2\Omega^i(M,g)$ respectively on $\mathcal{H}^i_{abs}(M,g)$, $\mathcal{H}^{i}_{rel}(M,g)$ and $\mathcal{H}^i_{min}(M,g)$. Then:\\ $\pi_{rel,i}\circ \pi_{abs,i}= \pi_{min,i}=\pi_{abs,i}\circ \pi_{rel,i}.$
\item $\mathcal{H}^i_{rel}(M,g)=\mathcal{H}_{min}^i(M,g)\oplus(\overline{ran(d_{max,i})}\cap \mathcal{H}^i_{rel}(M,g))$
\item $\overline{ran(d_{max,i})}=(\overline{ran(d_{max,i})}\cap \mathcal{H}^i_{rel}(M,g))\oplus \overline{ran(d_{min,i})}\oplus (\overline{ran(d_{max,i})}\cap \overline{ran(\delta_{max,i})})$
\end{enumerate}
Finally, if $(L^2\Omega^i(M,g),d_{max,i})$ or equivalently $(L^2\Omega^i(M,g),d_{min,i})$ is a Fredholm complex then there exists always an injective map $$\mathcal{H}^i_{min}(M,g)\rightarrow \im(H^i_{2,min}(M,g)\rightarrow H^{i}_{2,max}(M,g)).$$ Moreover the previous four equivalent conditions become:
\begin{enumerate}
\item $\mathcal{H}_{min}^{i}(M,g)\cong \im(H^i_{2,min}(M,g)\rightarrow H^{i}_{2,max}(M,g))$
\item $\mathcal{H}^i_{abs}(M,g)=\mathcal{H}_{min}^i(M,g)\oplus(ran(\delta_{max,i})\cap \mathcal{H}^i_{abs}(M,g))$
\item Let $\pi_{abs/rel/min,i}:L^2\Omega^i(M,g)\rightarrow \mathcal{H}_{ abs/rel/min}^i(M,g)$ be the orthogonal projections  of $L^2\Omega^i(M,g)$ respectively on $\mathcal{H}^i_{abs}(M,g)$, $\mathcal{H}^{i}_{rel}(M,g)$ and $\mathcal{H}^i_{min}(M,g)$. Then:\\ $\pi_{rel,i}\circ \pi_{abs,i}= \pi_{min,i}=\pi_{abs,i}\circ \pi_{rel,i}.$
\item $\mathcal{H}^i_{rel}(M,g)=\mathcal{H}_{min}^i(M,g)\oplus(ran(d_{max,i})\cap \mathcal{H}^i_{rel}(M,g))$
\item $ran(d_{max,i})=(ran(d_{max,i})\cap \mathcal{H}^i_{rel}(M,g))\oplus ran(d_{min,i})\oplus (ran(d_{max,i})\cap ran(\delta_{max,i}))$
\end{enumerate}
\end{prop}

\begin{proof}
Clearly it is enough to prove just  the first part of the proposition. The second part follows by the first part of the proposition and by the fact that   if $(L^2\Omega^i(M,g),d_{max/min,i})$ is a Fredholm complex then $d_{max/min,i}$ has closed range.
Let $\pi_{1,i}:\mathcal{H}^i_{rel}(M,g)\rightarrow \mathcal{H}^i_{abs}(M,g)$,   $\pi_{4,i}:\mathcal{H}^i_{abs}(M,g)\rightarrow \mathcal{H}^i_{rel}(M,g)$ be defined as in the proof of Theorem \ref{duality}. Moreover, by Prop. \ref{qwwq}, we know that $(\pi_{1,i})^*=\pi_{4,i}$ and analogously $(\pi_{1,i})^*=\pi_{4,i}$. By the proof of Theorem \ref{duality} we know that $\pi_{1,i}(\mathcal{H}^i_{rel}(M,g))\cong \im(\overline{H}^i_{2,min}(M,g)\rightarrow \overline{H}^{i}_{2,max}(M,g))$. Clearly, by the fact that $\mathcal{H}^i_{min}(M,g)=\mathcal{H}^i_{abs}(M,g)\cap \mathcal{H}^i_{rel}(M,g)$, we get that $\mathcal{H}^i_{min}(M,g)\subset \pi_{1,i}(\mathcal{H}^i_{rel}(M,g))$ and so we got \eqref{cesemoquasi}.\\ Now we pass to show that $1)\Rightarrow 2)$. As recalled above we know that $\pi_{1,i}(\mathcal{H}^i_{rel}(M,g))\cong \im(\overline{H}^i_{2,min}(M,g)\rightarrow \overline{H}^{i}_{2,max}(M,g))$ and that $\mathcal{H}^i_{min}(M,g)=\mathcal{H}^i_{abs}(M,g)\cap \mathcal{H}^i_{rel}(M,g)$; therefore using $1)$ we get that $\mathcal{H}^i_{min}(M,g)=\pi_{1,i}(\mathcal{H}^i_{rel}(M,g))$. This implies that   $(\mathcal{H}^i_{min}(M,g))^{\bot}\cap \mathcal{H}^i_{abs}(M,g)=(\pi_{1,i}(\mathcal{H}^i_{rel}(M,g)))^{\bot}\cap \mathcal{H}^i_{abs}(M,g)=Ker(\pi_{4,i})=$ $(\overline{ran(\delta_{max,i})}\cap \mathcal{H}^i_{abs}(M,g))$ and this completes the proof of the first implication.\\Now suppose that $2)$ is satisfied. Then it is immediate  that $\pi_{rel,i}\circ \pi_{abs,i}=\pi_{min,i}$ and therefore it is an easy consequence that also $\pi_{abs,i}\circ \pi_{rel,i}=\pi_{min,i}$.  Moreover it is still immediate that $3) \Rightarrow 4)$ because in this case $\pi_{4,i}(\mathcal{H}^i_{abs}(M,g))=\mathcal{H}^i_{min}(M,g).$
Now we want to show that $4)\Rightarrow 5).$ Clearly $\mathcal{H}^i_{min}(M,g)$ is orthogonal to $\overline{ran(\delta_{max,i})}$ and to $\overline{ran(d_{max,i})}$. This implies that the range of the orthogonal  projection of $\overline{ran(d_{max,i})}$ onto $\mathcal{H}^i_{rel}(M,g)$ is just the intersection  $\mathcal{H}^i_{rel}(M,g)\cap \overline{ran(d_{max,i})}$. From this we get  that also   the range of the orthogonal of projection of $\overline{ran(d_{max,i})}$ onto  $\overline{ran(\delta_{max,i})}$ is just the intersection  $\overline{ran(d_{max,i})}\cap \overline{ran(\delta_{max,i})}$ and therefore the implication $4)\Rightarrow 5)$ is proved. 
Finally, if $5)$ holds, it is immediate to show that $\pi_{1,i}(\mathcal{H}^i_{rel}(M,g))=\mathcal{H}^i_{min}(M,g)$ and this, using the fact that $\pi_{1,i}(\mathcal{H}^i_{rel}(M,g))\cong \im(\overline{H}^i_{2,min}(M,g)\rightarrow \overline{H}^{i}_{2,max}(M,g))$ implies $1)$. This completes the proof of the proposition.
\end{proof}

 \vspace{1 cm}

Institut f\"ur Mathematik, Humboldt-Universit\"at zu Berlin, Germany.

\texttt{E-mail address}: francescobei27@gmail.com bei@math.hu-berlin.de

\end{document}